\documentclass[11pt]{amsart}
\usepackage[english]{babel}
\usepackage{hyperref}
\usepackage{setspace}
\usepackage{url}
\usepackage{xcolor}
\usepackage{mathtools}
\usepackage{amssymb}
\usepackage{amsfonts}
\usepackage{comment}
\usepackage{amstext}
\usepackage{amsmath}
\usepackage{amsthm}
\usepackage[left=3cm, right=3cm, top=2.5cm, bottom=2.5cm]{geometry}
\newtheorem{theorem}{Theorem}[section]
\newtheorem{Definition}{Definition}[section]
\newtheorem{Corollary}{Corollary}[section]
\newtheorem{prop}{Proposition}[section]

\newtheorem{Remark}[theorem]{Remark}
\newcommand{\m}{\mu}
\newcommand{\p}{\phi}
\newcommand{\e}{\eta}
\newcommand{\la}{\lambda}
\newcommand{\D}{\Delta}
\newcommand{\ep}{\epsilon}

\newtheorem*{theorem*}{Theorem}
\newtheorem*{Corollary*}{Corollary}

\numberwithin{equation}{section}

\title[The Heat Kernel on an Invariant Subspace of 1-forms on $\mathbb{H}^{n}$]{ The Heat Kernel of the Hodge Laplacian on an Invariant Subspace of 1-forms on the Hyperbolic Space $\mathbb{H}^{n}$}

\author{PETROS SIASOS} 

\address{PETROS SIASOS, DEPARTMENT OF MATHEMATICS, UNIVERSITY OF THESSALY, 3rd km Old National Road Lamia – Athens, 35100, Lamia, GREECE}

\email{psiasos@uth.gr}

\begin{document}

\keywords{Heat kernel, Laplacian on forms, real hyperbolic space, Millson formula, uniform estimates. This work was partially supported by the University of Cyprus Internal Grant of Nelia Charalambous.}
\date{\today}
\subjclass[2000]{Primary 58J35; Secondary 53C21, 58J50}

\begin{abstract}
In this paper, we construct explicit formulas for the heat kernel of the Hodge Laplacian restricted to an invariant subspace of $1$-forms on the real hyperbolic space $\mathbb{H}^{n}$. Using warped product decompositions, the problem is reduced to component operators on specific subspaces, yielding precise closed-form expressions for both odd and even dimensions. We further establish Millson-type recurrence relations that connect heat kernels across differing dimensions. Based on these explicit representations, we derive sharp uniform estimates and long-time asymptotic behavior, extending classical scalar heat kernel bounds to $1$-forms on real hyperbolic spaces.
\end{abstract}

\maketitle
\tableofcontents

\section{Introduction}
On the real hyperbolic space $\mathbb{H}^n$,  due to the invariance of the isometry group, the heat kernel of the Laplace-Beltrami operator $\D_{\mathbb{H}^n}$ depends only on $\rho=d_{\mathbb{H}^n}(x,y)$ and its explicit formula is well-established (see \cite{davies1988heat}, \cite{grigor1998heat}).  In his PhD thesis \cite{Awonusika}, Awonusika recomputed this formula using the Green's function of $\D_{\mathbb{H}^n}$  and the generalized Mehler-Fock transform. Specifically, for $t > 0$, $0 < \rho < \infty$, the heat kernel $K_{\mathbb{H}^n}(t,\rho)$ associated to $\Delta_{\mathbb{H}^n}$ is given by 
\begin{itemize}
\item[(a)] $n$ odd $n \ge 3$:
   \begin{equation} \label{6.40}
    K_{\mathbb{H}^n}(t,\rho) = \left( -\frac{1}{2\pi \sinh \rho} \frac{\partial}{\partial \rho} \right)^{\frac{n-1}{2}} e^{-\frac{(n-1)^2}{4}t} K_{\mathbb{H}^1}(t,\rho),
   \end{equation} where $
    K_{\mathbb{H}^1}(t,\rho) = \frac{1}{\sqrt{4\pi t}} e^{-\frac{\rho^2}{4t}}$ 

\item[(b)] $n$ even $n \ge 2$:
    \begin{equation}\label{6.41}
    K_{\mathbb{H}^n}(t,\rho) = \left( -\frac{1}{2\pi \sinh \rho} \frac{\partial}{\partial \rho} \right)^{\frac{n-2}{2}} e^{-\frac{n(n-2)}{4}t} K_{\mathbb{H}^2}(t,\rho), 
    \end{equation} where $ 
    K_{\mathbb{H}^2}(t,\rho) =\frac{\sqrt{2}}{(4\pi t)^{\frac{3}{2}}} e^{-\frac{t}{4}} \int_\rho^\infty \frac{u e^{-\frac{u^2}{4t}}}{\sqrt{\cosh u - \cosh \rho}} \, du, $ is the heat kernel in $\mathbb{H}^2$.
\end{itemize}

Extending these heat kernel formulations to differential forms constitutes a crucial development for understanding global topological and geometric invariants within spectral geometry and global analysis, with far-reaching implications for non-smooth and singular spaces. Over recent decades, analogous formulas for the Hodge Laplacian heat kernel on Riemannian manifolds have been explored from various perspectives. In his PhD thesis \cite{Pedon1997}, Pedon investigated the harmonic analysis of differential forms on real hyperbolic spaces $\Lambda^p (\mathbb{H}^n)$, employing group-theoretic tools, generalized spherical functions, and the Abel transform to construct explicit heat kernel formulas. Jones in his PhD thesis \cite{Jones2008},  defined the Mehler–Fock transform on 1-forms and used it to find a formula for the heat kernel over the real hyperbolic plane $\mathbb{H}^2 $, expressing it in a unified way using differential and Hodge star operators. More recently, W. Bauer, K. Furutani, and C. Iwasaki \cite{Heisenberg.forms} studied the heat kernel of the Hodge Laplacian on forms over the Heisenberg group $\mathbb{H}_{2n+1} $, obtaining exact expressions inductively by reducing the problem to systems of ordinary differential equations.  

Our aim in this paper is to construct an explicit analog formula for the Hodge Laplacian heat kernel restricted to an invariant subspace of 1-forms on the real hyperbolic space $\mathbb{H}^n$. Furthermore, we use this explicit formula to establish sharp heat kernel estimates on this invariant subspace, generalizing the bounds established in \cite{davies1988heat}.

To describe the content and main results of the paper, we introduce some notation. Let $\theta^1, \dots, \theta^{n-1}$ be co-closed $1$-forms on $\mathbb{S}^{n-1}$ belonging to the eigenspace $E_{2(n-2)}(\Lambda^1 \mathbb{S}^{n-1})$ such that  $\{\bar{\theta}^1, \dots, \bar{\theta}^{n-1}, d\rho\}$, where $\bar{\theta}^i = \sinh \rho \, \theta^i$, is a pointwise orthonormal set of 1-forms on $\mathbb{H}^n$. We consider the subspace 
$$
W:= \text{span} \{\bar{\theta}^1, \dots, \bar{\theta}^{n-1}, d \rho\}
$$
of $ \Lambda^1 (\mathbb{H}^n)$ and set 
$$
A_i:=\{h(\rho) \bar{\theta}^i:  h(\rho) \in C^\infty[0,\infty) \}
$$
and 
$$
B:=\{h(\rho) d\rho: h(\rho) \in C^\infty[0,\infty) \}.
$$
As we see, the heat kernel on 1-forms in the Hyperbolic space $\mathbb{H}^n$ for the subspace W reduces to the computation of the heat kernels of  $\vec{\Delta}^{(1)}_{\mathbb{H}^n}\vert_{ A_i}$ and $\vec{\Delta}^{(1)}_{\mathbb{H}^n}\vert_{ B}$, where $\vec{\Delta}^{(1)}_{\mathbb{H}^n}$ is the Hodge Laplacian acting on 1-forms.

\begin{theorem*} (see Corollary \ref{heatkern.1-forms}) 
The heat kernel $\vec{K}^{(1)}_{\mathbb{H}^n} \vert_{ A_i} (t,\rho)=K_{1,\mathbb{H}^n} \vert_{ A_i}  (t, \rho)  \sinh \rho \ \theta^i$ of  $\vec{\Delta}^{(1)}_{\mathbb{H}^n}  \vert_{ A_i}$ is given by 
 \begin{itemize}
     \item [(a)] $n$ odd, $n \geq 1$:
     \begin{equation*} 
K_{1,\mathbb{H}^n} \vert_{ A_i}  (t, \rho)  =
\frac{\sinh  \rho \; e^{-\frac{(n-3)^2}{4}t} }{(2\pi)^{\frac{n+2}{2}}2^\frac{3}{2}\sqrt{t}}  
\left(-\frac{1}{\sinh\rho}\frac{\partial}{\partial\rho}\right)^{\frac{n+1}{2}}e^{-\frac{\rho^2}{4t}},
\end{equation*}
\item [(b)] $n$ even, $n \geq 2$:
$$
K_{1,\mathbb{H}^n} \vert_{ A_i}  (t, \rho)  
= \frac{\sinh \rho\; e^{-\frac{(n-3)^2}{4}t}}{8\sqrt{2}\pi^{\frac{3}{2}}t^{\frac{3}{2}}} \left(-\frac{1}{2\pi \sinh \rho} \frac{\partial}{\partial \rho}\right)^{\frac{n}{2}} \int_{\rho}^{\infty} \frac{u e^{-\frac{u^2}{4t}} \, du}{\sqrt{\cosh u - \cosh \rho}}.
$$
 \end{itemize}
Moreover, it satisfies the following Millson-type recurrence formula 
\begin{equation*}
   K_{1,\mathbb{H}^n} \vert_{ A_i}  (t, \rho)  = -\frac{e^{-(n-4)t}}{2\pi} \frac{\partial}{\partial \rho} \left[ \frac{K_{1, \mathbb{H}^{n-2}} \vert_{ A_i} (t, \rho)}{\sinh \rho} \right].
\end{equation*}
\end{theorem*}
As a direct consequence of the explicit formulas above, we obtain the following sharp bounds and asymptotic behavior.

\begin{Corollary*}(see Corollary \ref{D-V EST. 1-FORMS}) 
If $n \geqslant 1$ is any integer, then
\begin{equation*}
K_{1,\mathbb{H}^n} \vert_{ A_i} (t,\rho) \sim t^{-\frac{n+2}{2}} \rho \, (1+\rho+t)^{\frac{n-1}{2}}   \cdot e^{ -\frac{(n-3)^2}{4}t - \frac{\rho^2}{4t} - \frac{n-1}{2}\rho },
\end{equation*}
uniformly for $0 \leqslant \rho < \infty$ and $0 < t < \infty$. Also, for each fixed $\rho > 0$ and $n \geqslant 1$,$$K_{1,\mathbb{H}^n}\vert_{ A_i} (t,\rho) \sim t^{-\frac{3}{2}} e^{-\frac{(n-3)^2}{4} t} \quad \text{as } t \to \infty.$$
\end{Corollary*}

Having established the results for $\vec{\Delta}^{(1)}_{\mathbb{H}^n}\vert_{ A_i}$, the results for $\vec{\Delta}^{(1)}_{\mathbb{H}^n}\vert_{ B}$  follow from the following observation. The heat kernel $\vec{K}^{(1)}_{\mathbb{H}^n} \vert_{ B} (t,\rho)=K_{1,\mathbb{H}^n}  \vert_{ B} (t, \rho)\; d \rho $ of  $\vec{\Delta}^{(1)}_{\mathbb{H}^n}  \vert_{ B}$  satisfies 
$$
K_{1,\mathbb{H}^n} \vert_{ B}  (t, \rho)  =e^{-(n-2)t} K_{1,\mathbb{H}^n} \vert_{ A_i}  (t, \rho)
$$
(see Remark \ref{rem. 4.2}). Thus, the above Theorem and its Corollary for $\vec{\Delta}^{(1)}_{\mathbb{H}^n}  \vert_{ B}$  are obtained by replacing $e^{-\frac{(n-3)^2}{4}t}$ with $e^{-\frac{(n-1)^2}{4}t}$in the formulas of $K_{1,\mathbb{H}^n} \vert_{ A_i}  (t, \rho)$ and the Millson formula takes the following form
\begin{equation*}
   K_{1,\mathbb{H}^n} \vert_{ B}  (t, \rho)  = -\frac{1}{2\pi} \frac{\partial}{\partial \rho} \left[ \frac{K_{1, \mathbb{H}^{n-2}} \vert_{ B} (t, \rho)}{\sinh \rho} \right].
\end{equation*}

The paper is structured as follows. In Section 2, we establish the necessary preliminaries, reviewing the definitions of the Hodge Laplacian and heat kernels on differential forms, as well as the geometric structure of real hyperbolic space \(\mathbb{H}^n\), its spherical functions, and spectral properties. In Section 3, we analyze the decomposition of the Hodge Laplacian on 1-forms over $\mathbb{H}^n$ using warped product formulas. We introduce the invariant subspace $W$ and express the decomposition of $\vec{\Delta}^{(1)}_{\mathbb{H}^n}\vert_{ W}$ in terms of $\vec{\Delta}^{(1)}_{\mathbb{H}^n}\vert_{ A_i}$ and $\vec{\Delta}^{(1)}_{\mathbb{H}^n}\vert_{ B }$. Section 4 reviews the scalar heat kernel on \(\mathbb{H}^n\) using the Mehler-Fock transform and develops the framework for constructing the requisite Green's functions. In Section 5, we apply these methods to explicitly determine the heat kernel on 1-forms on hyperbolic space \(\mathbb{H}^n\) for the subspace \(W\), utilizing a supersymmetric framework to connect the radial Laplacian with component operators and deriving the explicit heat kernel formulas and Millson-type formulas for both odd and even dimensions. Finally, the Appendix provides a review of the supersymmetric framework and auxiliary operator relations utilized in the spectral analysis.

\section*{Acknowledgements}
I would like to thank my advisor Nelia Charalambous for suggesting me to work on this problem and also for her guidance and support for the completion of this project. 

\section{Preliminaries}

\subsection{The Laplacian and heat kernel on forms}
Let $(M,g)$ be a Riemannian manifold which is orientable. We consider the space of $k$-forms on $M$, denoted by $ \Lambda^k (M)$, which are smooth sections of the bundle $\Lambda^k(T^*M) $. The Riemannian metric $g$ induces a pointwise inner product  $<,>$ on $\Lambda^k(M)$  and this inner product gives us a global inner product $(,)$ on $ \Lambda^k (M)$ by
$$
(a,b)= \int_M <a,b>dV_g
$$from which we define $ L^2( \Lambda^k(M))$.
\begin{Definition}
Let $n$ be the dimension of $M$. The Hodge star operator $*$ is defined as
\begin{align*}
* : \Lambda^k(M) &\rightarrow \Lambda^{n-k} (M)\\
a&\mapsto *a
\end{align*}where $*$a is the unique element of $\Lambda^{n-k} (M) $ s.t. $a \wedge b=<*a,b>dV_g$  for every $b \in \Lambda^{n-k} (M) $.
\end{Definition}
Let us note that the pointwise inner product on $\Lambda^k(M)$ is related to the Hodge star operator by $<a,b>dV_g= a \wedge *b$. We have the following operators on $k$-forms. $d^k :\Lambda^k (M) \to \Lambda^{k+1} (M)$  is the exterior differentiation and $ \delta^{k+1} : \Lambda^{k+1} (M) \to \Lambda^k (M)$ is the co-differential operator defined by $\delta^{k+1}=(-1)^{nk+1}*d \;*$, with $*$ the Hodge star operator and $n=\text{dim}M$. Note that $d$ is independent of the metric, but $\delta$ does depend on the metric since $*$ does. The Laplacian on $k$-forms on $(M,g)$ is defined as follows:
\begin{Definition} (\cite{rosenberg1997laplacian} p. 33)
The Laplacian on smooth k-forms on a $(M,g)$ is given by $\vec\Delta^{(k)}=\delta^{k+1}\circ d^k + d^{k-1}\circ \delta^k$.
\end{Definition}
\begin{Definition}(\cite{rosenberg1997laplacian} p.34)
A double form $\vec{p}(t,x,y) \in (0,\infty) \times \Lambda^k (M) \otimes \Lambda^k (M) $ which is smooth in $t,x,y$ is called the heat kernel for the Laplacian on k-forms if 
\begin{itemize}
\item[(i)] $ (\partial_t+\vec\Delta^k_x)\vec{p}(t,x,y)=0$ and

\item[(ii)] $\lim_{t \to 0^+} \int_M <\vec{p}(t,x,y),\omega(y)>_y dy = \omega (x)$ for every smooth with compact support k-form $\omega$.
\end{itemize}
\end{Definition}
Locally, if we have a chart $\{x^1, \dotsc, x^n\}$  on $(M,g)$, then $\{dx^{i_1} \wedge \dotsm  \wedge dx^{i_k}: 1\leq i_1< \cdots < i_k \leq n \}$ is a basis for $\Lambda^k (M)$. For simplicity, we will use the multi-index notation $I=(i_1,\dots,i_k)$ to express $ dx^{i_1} \wedge \dotsm  \wedge dx^{i_k}$ as $dx^I$. Thus, the heat kernel takes the form 
\begin{equation}
\vec{p}(t,x,y)= \sum_{I,J} f_{I,J} (t,x,y)dx^I\otimes dy^J, \label{eq1heatdec}
\end{equation}
locally, where $I,J$ run over a local basis of $k$-forms for some $f_{I,J}$. Thus, to determine the heat kernel on forms locally it is enough to determine the $f_{I,J}$ for a pointwise basis of $k$-forms. Let us note that $f_{I,J}$ form a rectangular matrix of dimension equal to $ \binom{n}{k}$ for the case of $k$-forms. So, for the special case of $ 1$-forms we have that the $f_{I,J}$ form an $n\times n $ - matrix. Finally, using (\ref{eq1heatdec})  we have the following expression for the integral in $(ii)$ 
\begin{multline}
    \int_M <\vec{p}(t,x,y),\omega(y)>_y dy= \\ \sum_I \left[ \int_M <\sum_J f_{I,J}(t,x,y)dy^J, \sum_K \omega_K(y)dy^K>dy\right] dx^I. \label{eq2form}
\end{multline}

As an example of a heat kernel on forms we will compute it for the case of $1$-forms over the Euclidean space $(\mathbb{R}^n, g_0)$. Let $\{x^1,\dots, x^n\}$  be the Cartesian coordinates of $\mathbb{R}^n$  and $\{dx^1,\dots,dx^n\}$ be the corresponding co-frame. The heat kernel with respect to this co-frame is written as
$$
\vec{p}(t,x,y)=\sum _{i,k} e_{i,k}(t,x,y)dx^i \otimes dy^k.
$$
We will show that $e_{i,k}=0$ for $i\neq k$ and $e_{i,i} $  are all equal with the heat kernel on functions. Any $1$-form $\omega$ can be written as $\omega=\sum_i f_i dx^i$. Since $\vec{\Delta}_{\mathbb{R}^n}^{(1)}(f dx^i)=(\Delta_{\mathbb{R}^n} f) dx^i$  we have that the heat equation
$$
    \begin{cases}
     \left( \partial_t +\vec{\Delta}^{(1)}_{\mathbb{R}^n}\right) \omega(t,x)=0 \\
     \;\omega (x, 0)= \omega_0
    \end{cases}
$$
where 
$$\omega(t,x)=\sum_i f_i(t,x)dx^i, \;\; \omega_0(x)=\sum_i f_i^0(x)dx^i$$
is equivalent to
$$
    \begin{cases}
     ( \partial_t +\Delta_{\mathbb{R}^n}) f_i(t,x)=0 \\
     \;f_i(x, 0)= f_i^0(x)
    \end{cases}, i=1,\dots , n.
$$
From the above it is clear that the initial condition $$\omega_0(x)= f_1^0(x)dx^1+0 \cdot  dx^2+\cdots +0 \cdot dx^n$$ has solution
$$\omega(t,x)=f_1(t,x)dx^1+0 \cdot  dx^2+\cdots +0 \cdot dx^n.$$ At the same time (\ref{eq2form}) gives
$$
\omega(t,x)=\sum_{i=1}^n\left[\sum_{k=1}^n \int_{\mathbb{R}^n}e_{i,k}(t,x,y)f_k^0(y) dy\right]dx^i=\sum_{i=1}^n\left[\int_{\mathbb{R}^n} e_{i,1}(t,x,y)f_1^0 (y)dy\right]dx^i.
$$
Thus, if we set 
$$ a_i=\int_{\mathbb{R}^n} e_{i,1}(t,x,y)f_1^0 (y)dy,$$ 
we have that $a_2=0, \dots, a_n=0$.  Since $f^0_1(x) \neq 0$ was arbitrary we get that  $e_{i,1}(t,x,y)=0$ for $i=2, \dots, n$. Similarly, one can prove $e_{i,2}(t,x,y)=0$ for $i=1, 3,\dots, n$ and for the remaining cases. Finally, from 
$$f_1(t,x)=\int_{\mathbb{R}^n} e_{1,1}(t,x,y)f^0_1(y) dy,$$ 
we have that  $ e_{1,1}(t,x,y)$ is equal to the heat kernel  on functions. The same holds for every $ e_{i,i}(t,x,y)$.

 Thus, for the case of 1-forms the heat kernel with respect to the standard basis on $\mathbb{R}^n$  it is a diagonal matrix where each diagonal element is the heat kernel on functions. Similarly, one can show that the heat kernel on $k$-forms on $\mathbb{R}^n$, can be described as a rectangular matrix of dimension $\binom{n}{k} $, with its entries being the the heat kernel on functions. So in Euclidean spaces, the heat kernel on forms is the ''same'' as in the heat kernel on functions.
\subsection{The Hyperbolic space $\mathbb{H}^n$}
There are four ways in which one can define the hyperbolic space $\mathbb{H}^n$, the Poincare half-space model, the Poincare ball model, The Beltrami-Klein model, and the Hyperboloid model. Actually, these four models are equivalent since they are isometric, as Riemannian manifolds (see Theorem 3.7 in \cite{lee2018introduction}).

We consider the Poincare half-space model of the hyperbolic space
\begin{equation*}
    \mathbb{H}^n = \{ w = (x, y) : x \in \mathbb{R}^{n-1}, y > 0 \}.
\end{equation*}
This is a complete Riemannian manifold for the metric
\begin{equation*}
    g_{\mathbb{H}^n} = \frac{dx_1^2 + dx_2^2 + \dots + dx_{n-1}^2 + dy^2}{y^2}
\end{equation*}
and the volume element is $$d\mu_{\mathbb{H}^n}(x) = \frac{dx_1 dx_2 \dots dx_{n-1} dy}{y^n}. $$
The Laplace-Beltrami operator is given by  
\begin{equation*}
    \Delta_{\mathbb{H}^n} = -y^2 \sum_{i=1}^{n-1} \frac{\partial^2}{\partial x_i \partial x_i} - y^2 \frac{\partial^2}{\partial y^2} + (2-n)y \frac{\partial}{\partial y}.
\end{equation*}
In geodesic polar coordinates $(\rho,\p)$, the above are expressed as follows:
$$
g_{\mathbb{H}^n}=d\rho^2+\sinh^2 \rho \;g_{S^{n-1}} (\phi),
$$
where $g_{S^{n-1}}$ denotes the Riemannian metric of the unit sphere $S^{n-1}$,
\begin{equation} \label{pol. coord.}
    d\mu_{\mathbb{H}^n} = \sinh^{n-1} \rho \, d\rho \,d\phi,
\end{equation}
and 
\begin{equation} \label{Laplacian Hyperbolic spaces}
    \Delta_{\mathbb{H}^n} = -\frac{\partial^2}{\partial \rho^2} - (n-1) \coth \rho \frac{\partial}{\partial \rho} + \frac{1}{\sinh^2 \rho} \Delta_{S^{n-1}}.
\end{equation}
For future reference, we set 
\begin{equation}\label{rad. Lapl.}
    \D_{\mathbb{H}^n}^{\textbf{rad}}  := -\frac{\partial^2}{\partial \rho^2} - (n-1) \coth \rho  \;\frac{\partial}{\partial \rho}
\end{equation}
for the radial part of Laplacian in hyperbolic space $\mathbb{H}^n$.
A direct computation shows that for $s \in \mathbb{C}$ the function $y^s$ satisfies
\begin{equation} \label{eignv eq.}
    \Delta_{\mathbb{H}^n} y^s= s(n-1-s)y^s.
\end{equation}
Also, we have that the spectrum of $ \Delta_{\mathbb{H}^n}$, $\sigma(\Delta_{\mathbb{H}^n})$, is 
$$
\sigma(\Delta_{\mathbb{H}^n})=\left[ \left( \frac{n-1}{2} \right)^2, \infty \right).
$$
For the Laplacian on $k$-forms, $\vec\Delta^{(k)}_{\mathbb{H}^n}$, we have that
\begin{equation} \label{eignv eq. f.}
    \vec\Delta^{(k)}_{\mathbb{H}^n} y^s=-s(s-n+2k+1)y^s.
\end{equation}
We also recall the spherical functions in $\mathbb{H}^n$ which are defined by (see Proposition 6.2 in \cite{Awonusika})
\begin{equation} \label{mathcalC}
\Phi_r^{\mathbb{H}^n}(\rho) := \mathcal{C}_{-\frac{n-1}{2}-ir}^{\frac{n-1}{2}}(\cosh \rho) := 2^{-\frac{2-n}{2}} \Gamma\left( \frac{n}{2} \right) \sinh^{\frac{2-n}{2}} \rho P_{-\frac{1}{2}+ir}^{\frac{2-n}{2}}(\cosh \rho).
\end{equation}
These functions are normalized such that $\Phi_r^{\mathbb{H}^n}(0) = 1$  and satisfy the radial eigenvalue problem for the Laplace-Beltrami operator, expressed as 
\begin{equation*}
\left[ \frac{\partial^2}{\partial \rho^2} + (n-1) \coth \rho \frac{\partial}{\partial \rho} \right] \Phi_r^{\mathbb{H}^n}(\rho) = s(s-n+1)\Phi_r^{\mathbb{H}^n}(\rho).
\end{equation*}
Here $P_\nu^\mu(z)$ (and $Q_\nu^\mu(z)$, which will be used later) are the associated Legendre functions, which are solutions of the associated Legendre equation 
\begin{equation}
\frac{\mathrm{d}}{\mathrm{d}z}\left[(1-z^2)\frac{\mathrm{d}p}{\mathrm{d}z}\right] + \left[\nu(\nu+1) - \frac{\mu^2}{1-z^2}\right]p = 0
\end{equation}
in the complex $z$-plane (see Chapter 14 in \cite{NIST:DLMF} and Chapter 3 in Volume I of \cite{Bateman.Vol1and2}). Furthermore, their asymptotic behavior as $\rho \to \infty$ is governed by the limit formula 
\begin{equation*}
\lim_{\rho \to \infty} e^{\left( \frac{n-1}{2} + ir \right) \rho} \Phi_r^{\mathbb{H}^n}(\rho) = \frac{2^{n-2} \Gamma\left( \frac{n}{2} \right) \Gamma(ir)}{\sqrt{\pi} \Gamma\left( \frac{n-1}{2} + ir \right)} = c(r),
\end{equation*}
where $c(r)$ denotes the Harish-Chandra $c$-function associated with $\mathbb{H}^n$ (see Theorem 6.1  in \cite{Awonusika}). Finally, we recall the definition of the Plancherel measure in $\mathbb{H}^n$ (see Definition 6.4 in \cite{Awonusika}), given by 
\begin{equation}\label{planc}
\sigma_{\mathbb{H}^n}(r) := |c(r)|^{-2} = [c(r)c(-r)]^{-1} = \frac{\pi \left| \Gamma\left( \frac{n-1}{2} + ir \right) \right|^2}{2^{2n-4} \Gamma\left( \frac{n}{2} \right)^2 |\Gamma(ir)|^2}.
\end{equation}
\section{Decomposition analysis of $\vec{\Delta}^{(1)}_{\mathbb{H}^n}$  on 1-forms }\label{section3}

In order to decompose the Laplacian on 1-forms $\vec{\Delta}^{(1)}_{\mathbb{H}^n}$  on the hyperbolic space $\mathbb{H}^n$, we use Donnelly's decomposition \cite{donnelly1981differential}(see also \cite{antoci2004spectrum}). Let $M= \mathbb R \times N$ be a product manifold, with the warped product metric 
$$
g=d\rho^2+f^2(\rho)g_{N},
$$
then the action of $\vec{\Delta}^{(k)}$ on $k$-forms on $M$ of the form $\omega=\omega_1+\omega_2 \wedge d\rho$, where $\omega_1$ and $\omega_2$ are respectively a $k$-form and a $(k-1)$-form on $N$ depending on $\rho$ is given in the next Proposition.

\begin{prop} \label{Delta expression} \cite{antoci2004spectrum}
Let $\omega_1=h_1(\rho)\e_1$, $\omega_2=h_2(\rho)\e_2$,  where $\e_1$ is a $k$-form on $N$ and $\e_2$ is a $(k-1)$-form on $N$, and $h_1,h_2$ are smooth functions of $\rho$. Then, the $k$-form $\omega=\omega_1+\omega_2 \wedge d\rho$ on $M$, satisfies
\begin{equation*}
    \begin{split}
        \vec\D^{(k)} \omega & =h_1 f^{-2}  \vec\D_N \e_1 +h_2f^{-2} ( \vec\D_N \e_2) \wedge d\rho +(-1)^k2 h_1 f' f^{-3}( \delta_N \e_1 )\wedge d\rho \\ 
        &+(-1)^k 2h_2 f' f^{-1} d_N \e_2 -[h_1''+(n-2k-1)h_1'f'f^{-1}] \e_1\\
        & - [h_2'' +(n-2k+1)(h_2f' f^{-1})']\e_2 \wedge d\rho.\\
    \end{split}
\end{equation*}
\end{prop}
In our study on the heat kernel on 1-forms in hyperbolic space $\mathbb{H}^n$, we focus on certain types of $1$-forms. Thus, we describe in detail the action of the Laplacian in these cases. 
\begin{Corollary} \label{Corollary 3.1} \cite{siasos2024p}
    \begin{itemize}
        \item [(a)] Let $\e$ be a co-closed $k$-eigenform over $N$ with $ \vec\D_N^{(k)} \e =\lambda \e$. Then, the $k$-form $\omega=h_1(\rho) \e$ over $M$ satisfies
        $$
         \vec\D^{(k)}(h_1(\rho) \e) = \D_1(h_1(\rho)) \e,
        $$ 
        where 
        $$
\Delta_1(h_1(\rho))=-\left[h_1''(\rho)+(n-2k-1)h_1'(\rho)\frac{f'(\rho)}{f(\rho)}\right]+\lambda \frac{h_1(\rho)}{f^2(\rho)}.
        $$
        \item [(b)] Let $\e$ be a closed $(k-1)$-eigenform over $N$ with $\vec\D_N^{(k-1)} \e =\lambda_0 \e$. Then, the $k$-form $\omega=h_2(\rho) \e \wedge d\rho$ over $M$ satisfies
        $$
         \vec\D^{(k)} (h_2(\rho) \;\e \wedge d\rho) = \D_2(h_2(\rho)) \; \e \wedge d\rho,
        $$
        where 
     \begin{equation*}
\Delta_2(h_2(\rho))=-\left[h_2''(\rho)+(n-2k+1)\left(h_2(\rho)\frac{f'(\rho)}{f(\rho)}\right)'\right]+\lambda_0 \frac{h_2(\rho)}{f^2(\rho)}.
\end{equation*}
    \end{itemize}
\end{Corollary}
In the special case of the hyperbolic space $\mathbb{H}^{n}= [0, \infty) \times S^{n-1}$ with metric $
g=d\rho^2+\sinh^2 \rho g_{S^{n-1}}$, for 1-forms, the above Corollary takes the following form.
\begin{Corollary}\label{Corollary}
    \begin{itemize}
        \item [(a)] Let $\e$ be a co-closed $1$-eigenform over $S^{n-1}$ with $\vec{\Delta}^{(1)}_{S^{n-1}} \e =\lambda \e$. Then, the $1$-form $\omega=h_1(\rho) \e$ over $\mathbb{H}^{n}$ satisfies
        $$
         \vec{\Delta}^{(1)}_{\mathbb{H}^n}(h_1(\rho) \e) = \D_1(h_1(\rho)) \e,
        $$ 
        where 
        $$
\Delta_1(h_1(\rho))=-h_1''(\rho)-(n-3)\coth \rho \;h_1'(\rho)+\lambda \frac{h_1(\rho)}{\text{sinh}^2 \rho}.
        $$
    \item [(b)] Let $\e$ be a closed $0$-eigenform over $S^{n-1}$ with $\Delta_{S^{n-1}} \e =\lambda_0 \e$. Then, the $1$-form $\omega=h_2(\rho) \e \wedge d\rho$ over $\mathbb{H}^{n}$ satisfies
        $$
         \vec{\Delta}^{(1)}_{\mathbb{H}^n}(h_2(\rho) \;\e \wedge d\rho) = \D_2(h_2(\rho)) \; \e \wedge d\rho,
        $$
        where 
     \begin{equation*}
\Delta_2(h_2(\rho))=-h_2''(\rho)-(n-1)\coth \rho\; h_2'(\rho)+ (n-1+\lambda_0)\frac{ h_2(\rho)}{{\sinh }^2 \rho}.
\end{equation*}
    \end{itemize}
\end{Corollary}

Now, we specify particular values in the above Corollary to define the differential operators used in what follows.
\begin{Remark}\label{eq1formszero} For the special case $\eta=1$ in $\omega=h_2(\rho) \e \wedge d\rho$ in Corollary \ref{Corollary} since $\la_0=0$ we have \begin{equation*}
\Delta_2(h(\rho))=-h''(\rho)-(n-1)\coth \rho \;h'(\rho)+ (n-1)\frac{ h(\rho)}{\sinh^2 \rho}. 
\end{equation*}
\end{Remark}

\begin{Remark} \label{eq2formszero}
For the special case $\la= 2(n-2)$ in (a), a straightforward computation gives
$$\Delta_{1}(h(\rho) \sinh \rho) = \tilde{\Delta}_{1} (h(\rho))$$ where $\tilde{\Delta}_{1}$ is a one-dim differential operator defined by 
\begin{equation*}
    \tilde{\Delta}_{1}f(\rho) := -f(\rho)'' - (n-1) \coth \rho \, f(\rho)' + \left[ -(n-2)+  \frac{ n-1}{\sinh^{2} \rho} \right] f(\rho).
\end{equation*}
\end{Remark}
 
Let $\theta^1, \dots, \theta^{n-1}$ be co-closed $1$-forms on $\mathbb{S}^{n-1}$ belonging to the eigenspace $E_{2(n-2)}(\Lambda^1 \mathbb{S}^{n-1})$ such that they form a local pointwise orthonormal coframe \cite{Gallot-Meyer}, \cite{Ikeda-Taniguchi}. Then, $\{\bar{\theta}^1, \dots, \bar{\theta}^{n-1}, d\rho\}$, where $\bar{\theta}^i = \sinh \rho \, \theta^i$, forms a local orthonormal frame of 1-forms on $\mathbb{H}^n$.

Now we define the subspace 
$$
W:= \text{span} \{\bar{\theta}^1, \dots, \bar{\theta}^{n-1}, d \rho\}
$$
of $ \Lambda^1 (\mathbb{H}^n)$ for which we will obtain the heat kernel of $\vec{\Delta}^{(1)}_{\mathbb{H}^n}$ in 1-forms and compute its action. For any  $\omega \in W$ let $h(\rho), h_1(\rho),\dots, h_{n-1}(\rho)$ be smooth functions such that 
\[
\omega = h_1(\rho) \bar{\theta}^1 + \dots + h_{n-1}(\rho) \bar{\theta}^{n-1} + h(\rho) d\rho.
\]
Setting
$$
A_i:=\{h(\rho) \bar{\theta}^i:  h(\rho) \in C^\infty[0,\infty) \}
$$
and 
$$
B:=\{h(\rho) d\rho: h(\rho) \in C^\infty[0,\infty) \},
$$
it follows by linearity of $\Delta_{\mathbb{H}^n}^{(1)}$ and Corollary \ref{Corollary} that
$$
W= A_1\oplus\cdots \oplus A_
{n-1}\oplus B
$$
and
\begin{equation}\label{dec}
\vec{\Delta}^{(1)}_{\mathbb{H}^n} \vert_{W}=\bigoplus_{i=1}^{n-1}\vec{\Delta}^{(1)}_{\mathbb{H}^n}\vert_{ A_i} \oplus\vec{\Delta}^{(1)}_{\mathbb{H}^n}\vert_B.
\end{equation}
For future reference, we note that by Remarks \ref{eq1formszero}, \ref{eq2formszero}, the component operators in (\ref{dec}) are given by
\begin{equation}\label{dec'}
  \vec{\Delta}^{(1)}_{\mathbb{H}^n}\vert_{ A_i} \left(h(\rho)  \bar\theta^i\right) = \tilde{\D}_1 \left(h(\rho)  \right)\theta^i \; \text{and} \;\vec{\Delta}^{(1)}_{\mathbb{H}^n}\vert_B \left(h(\rho)d\rho\right)=\D_2\left(h(\rho) \right)d\rho.
\end{equation}
In what follows, the above analysis is used to show that computing the heat kernel for the subspace W of 1-forms in the hyperbolic space $\mathbb{H}^n$  reduces to calculating the heat kernels of  $\vec{\Delta}^{(1)}_{\mathbb{H}^n}\vert_{ A_i}$ and $\vec{\Delta}^{(1)}_{\mathbb{H}^n}\vert_{ B}$.

\section{The heat kernel on functions on $\mathbb{H}^n$}

In \cite{Awonusika}, Awonusika recomputed the well-known heat kernel of the Laplacian on functions on $\mathbb{H}^n$ using its Green's function and the generalized Mehler-Fock transform (see Theorem 6.7 in \cite{Awonusika}). For  the generalized Mehler-Fock transform, see also \cite{Guseinov}. In this section, we state and prove these results, as they will serve as the scalar prototypes for their form-valued analogs in the next section.

\begin{theorem}[Theorem 6.7 in \cite{Awonusika}] \label{Thm 6.7 (0-forms)}
Let  $\Delta'_{\mathbb{H}^n}: = \Delta_{\mathbb{H}^n} + \frac{(n-1)^2}{4}$ be the shifted Laplacian in $\mathbb{H}^n$ and define the operator $h\left( \sqrt{\Delta'_{\mathbb{H}^n}} \right)$ via the functional calculus for a suitable function $h$  satisfying
\begin{equation*}
h\left( \sqrt{\Delta'_{\mathbb{H}^n}} \right) f(w) = \int_{\mathbb{H}^n} K_{\mathbb{H}^n}(w, w') f(w') \, d\mu_{\mathbb{H}^n}(w').
\end{equation*}
Then, its kernel is given by
\begin{equation}\label{6.54}
\begin{aligned}
 K_{\mathbb{H}^n}(w, w') =K_{\mathbb{H}^n}(\rho) &= \frac{2^{n-3} \Gamma \left( \frac{n}{2} \right)}{\pi^{\frac{n+2}{2}}} \int_0^\infty h(r) \sigma_{\mathbb{H}^n}(r) \mathcal{C}_{-\frac{n-1}{2}-ir}^{\frac{n-1}{2}}(\cosh \rho) \, dr,
\end{aligned}
\end{equation} where $\rho=d(w,w')$. Also, the kernel $K_{\mathbb{H}^n}(w, w')$ is the generalized Mehler-Fock integral formula for $\mathbb{H}^n$, satisfying the Millson-type recurrence formula
\begin{equation}\label{6.55}
K_{\mathbb{H}^{n+2}}(\rho) = -\frac{1}{2\pi \sinh \rho} \frac{\partial}{\partial \rho} K_{\mathbb{H}^n}(\rho).
\end{equation} Moreover, if we take $h(r)=e^{-r^2t}e^{-\frac{(n-1)^2}{4}t}$ in (\ref{6.54}) we get the heat kernel of $\D_{\mathbb{H}^n}$ (see (\ref{6.40}) and (\ref{6.41})).
\end{theorem} 
In the proof of Theorem \ref{Thm 6.7 (0-forms)}, we need the explicit formula for the Green function for $\Delta_{\mathbb{H}^n}$. This formula is well-known (see Theorem 6.6 in \cite{Awonusika}); nevertheless, its  complete proof is not widely accessible, so we provide it here for completeness. This proof will also serve as a guide for the proof of the Green function for  $\vec{\Delta}^{(1)}_{\mathbb{H}^n}\vert_{ A_i}$.

\begin{theorem}\label{Thm G. H^n 0-f.}
Let $n\geq 2$ be an integer. The Green function of $\Delta_{H^n}$ is equal to
\begin{equation*}
   R_{\mathbb{H}^n}( w, w';s) = c_n(s) \frac{1}{\sigma^s} F\left(s, s - \frac{n}{2}+1, 2s - n+2), \frac{1}{\sigma}\right),
\end{equation*}
where $\sigma := \cosh^2(d(w, w')/2)$ and 
\begin{equation*}
    c_n(s) = \frac{1}{2^{n} \pi^{n/2}} \frac{\Gamma(s) \Gamma\left(s - \frac{n}{2} + 1\right)}{\Gamma(2s - n+2))}.
\end{equation*}
\end{theorem}
\begin{proof}
The proof follows the approach of Section 4.1 in \cite{Borthwickbook2ed}. Note, however, that Borthwick computed the Green function for the Hyperbolic plane $\mathbb{H}^2$ only, rather than for the general case $\mathbb{H}^n$. We thus present the general case; to our knowledge, this proof is new.

In general, the resolvent of a positive self-adjoint operator $A$ is the bounded operator   $(A-z)^{-1}$, with $z \in \mathbb{C} - [0,\infty)$ on $L^2(\mathbb{H}^n)$.  For  $\Delta_{\mathbb{H}^n}$, equation (\ref{eignv eq.}) suggests taking $z=s(n-1-s)$, so the resolvent is 
$$
R_{\mathbb{H}^n}(s) := (\Delta - s(n - 1 - s))^{-1} , 
$$
 for  $\text{Re } s > \frac{n-1}{2} $ with  $ s \notin [\frac{n-1}{2}, n-1]$.

It is a well known fact that the Green function of $\Delta_{\mathbb{H}^n}$ coincides  with the integral kernel $R_{\mathbb{H}^n}( w,w';s) $ of  $ R_{\mathbb{H}^n}(s)$. By the definition of $ R_{\mathbb{H}^n}(s)$, we have 

\begin{equation} \label{fund. sol. eq.}
(\Delta_{\mathbb{H}^n} - s(n - 1 - s))R_{\mathbb{H}^n}( w, w';s) = y^n \delta(w - w').
\end{equation}

By the homogeneity of the hyperbolic space $\mathbb{H}^n$, we can write $R_{\mathbb{H}^n}( w, w';s)= f_s(d(w,w'))$ for some function $f_s$. Considering geodesic polar coordinates $ (\rho, \p) $ centered on $w'$, we have $\rho=d(w,w')$, and taking into account the radial Laplacian $\D_{\mathbb{H}^n}^{\textbf{rad}}$,  the homogeneous equation corresponding to (\ref{fund. sol. eq.}) is

\begin{equation} \label{(4.8)}
f_s''(\rho) + (n-1) \coth(\rho) f_s'(\rho) + s(n-s-1) f_s(\rho) = 0,
\end{equation}
for $\rho>0$. 
We make the transformation $ h_s(u) =f_s(\rho)$ given by $u :=\cosh{\rho}$ on (\ref{(4.8)}). For simplicity, from now on we will write $h(u)$ for $h_s(u)$ and (\ref{(4.8)}) becomes

\begin{equation*}
    (1 - u^2) h''(u) - n u h'(u) - s(n - 1 - s) h(u) = 0,
\end{equation*}
which is the associated Legendre equation.  Since the solution to the above equation should be regular at infinity when $\text{Re} s> n-1$, the solution should be 

\begin{equation} \label{sol.}
    h(u)=c_n(s) (u^2 - 1)^{\frac{2 - n}{4}} Q_{s - \frac{n}{2}}^{1 - \frac{n}{2}}(u)
\end{equation}
for some coefficient $c_n(s)$.

To compute the coefficient $c_n(s)$,we use the boundary condition as $u \to 1^+$.
First, we apply the asymptotic behavior of $ Q_\nu^\mu(u)$ as $u \to 1^+$ given by Eq. (11) on p. 163 in the first volume of \cite{Bateman.Vol1and2}:
\begin{equation}\label{Batem.(11)}
    Q_\nu^\mu(u) \sim 2^{-\frac{\mu}{2}-1} \frac{\Gamma(-\mu) \Gamma(\nu+\mu+1)}{\Gamma(\nu-\mu+1)} (1-u)^{\mu/2}, \quad \text{Re } \mu < 0.
\end{equation}
We have

$$Q_{-\frac{n}{2} + s}^{1 - \frac{n}{2}}(u) \sim \frac{2^{\frac{n-2}{4}}}{2} \cdot \Gamma\left(\frac{n}{2}-1\right) \frac{\Gamma(\nu+\mu+1)}{\Gamma(s)} \cdot \frac{1}{(1-u)^{\frac{n-2}{4}}}, \quad u \to 1^+,  $$
which yields

$$
 (u^2 - 1)^{\frac{2-n}{4}} Q_{-\frac{n}{2} + s}^{1 - \frac{n}{2}}(u) \sim -\frac{1}{2} \frac{\Gamma\left(\frac{n}{2}-1\right) \Gamma(\nu+\mu+1)}{\Gamma(s)} \cdot \frac{1}{(u-1)^{\frac{n-2}{2}}}, \quad u \to 1^+ .
$$
In terms of the initial variable $\rho$, the asymptotic behavior described above takes the following form
\begin{equation*}
   (\sinh^2 \rho)^{\frac{2-n}{4}} Q_{-\frac{n}{2} + s}^{1 - \frac{n}{2}}(\cosh \rho) \sim \frac{1}{2 \cdot 2^{\frac{2-n}{2}}} \frac{\Gamma(\frac{n}{2}-1) \Gamma(\nu+\mu+1)}{\Gamma(s)} \cdot \frac{1}{\rho^{n-2}}, \quad \rho \to 0^+.
\end{equation*}
Thus, the solution 
\begin{equation*}
   f_s(\rho) = c_n(s) \cdot (\sinh^2 \rho)^{\frac{2-n}{4}} Q_{-\frac{n}{2} + s}^{1 - \frac{n}{2}}(\cosh \rho) 
\end{equation*}
of (\ref{fund. sol. eq.})  has the following asymptotic behavior
\begin{equation} \label{bdr a. b. Q}
    f_s(\rho) \sim c_n(s) \frac{1}{2 \cdot 2^{-\frac{n-2}{2}}} \cdot \frac{\Gamma(\frac{n}{2}-1) \Gamma(\nu+\mu+1)}{\Gamma(s)} \cdot \frac{1}{\rho^{n-2}}, \quad \rho \to 0^+.
\end{equation}

 Next, we integrate (\ref{fund. sol. eq.}) over the ball $B(0, \ep)$ to obtain

 \begin{equation*}
      \int_0^\epsilon \int_{S^{n-1}} \left( \Delta_{\mathbb{H}^n} - s(n-s-1) \right) f_s(\rho) \sinh^{n-1} \rho \, d\rho d\theta = \int_0^\epsilon \int_0^{2\pi} \frac{\delta(\rho)}{\sinh^{n-1} \rho} \sinh^{n-1} \rho \, d\rho d\theta.
 \end{equation*}
After a straightforward calculation, we get

\begin{equation*}
    \int_0^\epsilon \int_{S^{n-1}} \left[ -\frac{1}{\sinh^{n-1} \rho} \left( \sinh^{n-1} \rho \, f_s'(\rho) \right)' - s(n-s-1) f_s(\rho) \right] \sinh^{n-1} \rho \, d\rho d\theta = 1
\end{equation*}
and finally
$$
-\omega_{n-1} \int_{0}^{\epsilon} (\sinh^{n-1} \rho f_s'(\rho))' d\rho - \omega_{n-1} s(n-s-1) \int_{0}^{\epsilon} f_s(\rho) \sinh^{n-1} \rho d\rho = 1,
$$
where $\omega_{n-1}$ is the surface area of the unit sphere $S^{n-1}$.

By imposing  $f_s \in L^2_{\text{loc}} $, the second term in the last equation vanishes as $\epsilon \rightarrow 0$, and the first term yields

\begin{equation*}
    \lim_{\epsilon \rightarrow 0} -\sinh^{n-1} {\ep} f_s'(\ep)= \frac{1}{\omega_{n-1}} .
\end{equation*}
After a straightforward calculation and using the identity  $\omega_{n-1} = \frac{\Gamma(\frac{n}{2})}{2\pi^{n/2}}$, we get
\begin{equation} \label{bdr a. b. Ball}
    f_{s}(\epsilon) \sim -\frac{\Gamma(\frac{n}{2})}{2\pi^{n/2}(2-n)} \cdot \frac{1}{\epsilon^{n-2}}, \quad \epsilon \to 0.
\end{equation}

Finally, combining (\ref{bdr a. b. Q}) and (\ref{bdr a. b. Ball}) with the identity  $\Gamma (\frac{n}{2})=\frac{n-2}{2}\Gamma (\frac{n-2}{2}) $, we find

\begin{equation}\label{c_n(s)}
    c_n(s) = -\frac{2^{\frac{2-n}{2}}}{2\pi^{n/2}} \cdot \frac{\Gamma(s)}{\Gamma(\nu+\mu+1)}.
\end{equation} 
Next, we use the expression for $Q_{\nu}^{\mu}(z)$ in terms of the hypergeometric equation (see Eq. (10) on p. 123, Eq (36) on p. 132 and Eq. (36) on p. 133 in the first volume of  \cite{Bateman.Vol1and2}):
\begin{equation} \label{Bateman10}
\begin{split}
e^{-i \mu \pi} Q_{\nu}^{\mu}(x) = & \ 2^{\nu} \frac{\Gamma(1+\nu) \Gamma(1+\nu+\mu)}{\Gamma(2+2\nu)} (x+1)^{\frac{\mu}{2} - \nu - 1} (x-1)^{-\frac{\mu}{2}} \\
& \cdot F\left(1+\nu-\mu, 1+\nu; 2+2\nu; \frac{2}{1+x}\right)
\end{split}
\end{equation}
and find 
\begin{equation*}
\begin{split}
Q_{s - \frac{n}{2}}^{1 - \frac{n}{2}}(u) = & \ 2^{s - \frac{n}{2}} \frac{\Gamma(s - \frac{n}{2} + 1) \Gamma(\nu + \mu + 1)}{\Gamma(2s - n + 2)} (u+1)^{\frac{n}{4} - \frac{1}{2} - s} (u-1)^{\frac{n-2}{4}} \\
& \cdot F\left(s, s - \frac{n}{2} + 1; 2s - n + 2; \frac{2}{1+u}\right)
\end{split}.
\end{equation*}
Substituting the above equation and (\ref{c_n(s)})  into (\ref{sol.}),  we get
\begin{equation*}
\begin{split}
h_s(u) = & -\frac{ 2^{\frac{2-n}{2}} \cdot 2^{s - \frac{n}{2}}}{2\pi^{n/2}} \cdot \frac{\Gamma(s) \Gamma(s - \frac{n}{2} + 1)}{\Gamma(2s - n + 2)} (u+1)^{\frac{n}{4} - \frac{1}{2} - s} (u-1)^{\frac{n-2}{4}} \\[10pt]
& \cdot F\left(s, s - \frac{n}{2} + 1; 2s - n + 2; \frac{2}{1+u}\right)
\end{split}.
\end{equation*}
By a straightforward calculation and changing the variable $\sigma:= \cosh^2\left(\frac{r}{2}\right)$ we obtain
\begin{equation*}
f_{s}(\sigma) = -\frac{1}{2^n\pi^{n/2}} \cdot \frac{\Gamma(s) \Gamma(s - \frac{n}{2} + 1)}{\Gamma(2s - n + 2)} \left( \frac{1}{\sigma} \right)^{s} \cdot F\left(s, s - \frac{n}{2} + 1; 2s - n + 2; \frac{1}{\sigma}\right).
\end{equation*}

\end{proof}

\textit{Proof of Theorem \ref{Thm 6.7 (0-forms)}:} 

Using Theorem \ref{Thm G. H^n 0-f.} for $s=\frac{n-1}{2}-ir$ and $s=\frac{n-1}{2}+ir$, we have

\begin{equation}\label{Green.differ.}
\begin{split}
R_{\mathbb{H}^n}\left(\sigma; \frac{n-1}{2} - ir\right) - R_{\mathbb{H}^n}\left(\sigma; \frac{n-1}{2} + ir\right) =    \\ \frac{1}{ \pi^{n/2} 2^{n} } \left[ \left(\frac{1}{\sigma}\right)^{\frac{n-1}{2} - ir} \cdot 
\frac{\Gamma (\frac{n-1}{2} - ir)\Gamma(\frac{1}{2}-ir) }{\Gamma(1-2ir)} \cdot F\left(\frac{n-1}{2} - ir,\frac{1}{2}-ir; 1-2ir; \frac{1}{\sigma}\right) \right.\\ -\left.
\left(\frac{1}{\sigma}\right)^{\frac{n-1}{2} + ir} \cdot 
\frac{\Gamma (\frac{n-1}{2} + ir)\Gamma(\frac{1}{2}+ir) }{\Gamma(1+2ir)} \cdot F\left(\frac{n-1}{2} + ir, \frac{1}{2}+ir; 1+2ir; \frac{1}{\sigma}\right)   \right]
\end{split}.
\end{equation}
Awonusika found the following expression for the above difference 

\begin{equation}\label{4.15}
    \begin{aligned}
        R_{\mathbb{H}^n}\left(\sigma; \frac{n-1}{2} - ir\right) - R_{\mathbb{H}^n}\left(\sigma; \frac{n-1}{2} + ir\right) = &\frac{\Gamma\left(\frac{n-1}{2} - ir\right) \Gamma\left(\frac{n-1}{2} + ir\right) \Gamma\left(\frac{1}{2} + ir\right) \Gamma\left(\frac{1}{2} - ir\right)}{2^n \pi^{\frac{n}{2}} (-2ir) \Gamma(-2ir) \Gamma(2ir) \Gamma\left(\frac{n}{2}\right)} \cdot\\
        & F\left(\frac{n-1}{2} + ir, \frac{n-1}{2} - ir; \frac{n}{2}; 1 - \sigma\right)
    \end{aligned} 
\end{equation}
By the homogeneity of the hyperbolic space and by considering geodesic polar coordinates $ (\rho, \p)$ centered at $w'$ with $\rho=d(w,w')$, the above equation is equal to

\begin{equation}\label{6.56}
\begin{aligned}
        R_{\mathbb{H}^n}\left( \cosh^2\left(\frac{\rho}{2}\right); \frac{n-1}{2} - ir \right) - R_{\mathbb{H}^n}\left( \cosh^2\left(\frac{\rho}{2}\right); \frac{n-1}{2} + ir \right)
&= \frac{i}{\pi^{\frac{n}{2}}r} 2^{n-3} \Gamma\left(\frac{n}{2}\right) \cdot\\ &\sigma_{\mathbb{H}^n}(r) \mathcal{C}^{\frac{n-1}{2}}_{-\frac{n-1}{2}-ir}(\cosh \rho)
\end{aligned},
\end{equation}
where $\mathcal{C}^{\frac{n-1}{2}}_{-\frac{n-1}{2}-ir}(\cosh \rho) $ is defined in (\ref{mathcalC}).

Now, let $E_\la, \la \in \mathbb{R}$, be the spectral projection for the self-adjoint Laplacian $\Delta'_{\mathbb{H}^n} = \Delta_{\mathbb{H}^n} + \frac{(n-1)^2}{4}$: 

\begin{equation*}
\Delta_{\mathbb{H}^n}' f = \int_0^\infty \la \, dE_\la f, \quad f \in L^2(\mathbb{H}^n).
\end{equation*}
Using Stone's formula (see Eq. (A.15) in \cite{Awonusika}), we have

\begin{equation*}
dE_\la f(w) = \frac{1}{2\pi i} [R_{\la+i0} - R_{\la-i0}] f(w) \, d\la,
\end{equation*}
where $R_\la$ is the resolvent operator in $\mathbb{H}^n$ defined by

\begin{equation*}
(R_\la f)(w) = \int_{\mathbb{H}^n} R_{\mathbb{H}^n}'(w, w'; \la) f(w') \, d\mu_{\mathbb{H}^n}(w'), 
\end{equation*}
with

\begin{equation*}
R_{\mathbb{H}^n}'(w, w'; \la) = 
\begin{cases} 
R_{\mathbb{H}^n}(\sigma; \frac{n-1}{2} + i\sqrt{\la}), & \text{Im } \la < 0, \\
R_{\mathbb{H}^n}(\sigma; \frac{n-1}{2} - i\sqrt{\la}), & \text{Im } \la > 0.
\end{cases} 
\end{equation*}
By the above, we have

\begin{equation*}
(E_\la f)(w) =  \int_{\mathbb{H}^n} \frac{1}{2\pi i}  \int_0^\la \left[ R_{\mathbb{H}^n}'(w, w'; r + i0) - R_{\mathbb{H}^n}'(w, w'; r - i0) \right]  dr\;f(w') \, d\mu_{\mathbb{H}^n}(w').
\end{equation*}
Let us denote its kernel by 
$$
\pi_{\la}(w, w') = \frac{1}{2\pi i} \int_{0}^{\la} \left[ R_{\mathbb{H}^n}'(w, w'; r + i0) - R_{\mathbb{H}^n}'(w, w'; r - i0) \right] dr, $$
then using  (\ref{6.56}),  we compute
\begin{equation}\label{pirho}
\begin{aligned} 
\pi_{\la}(w, w')=\pi_{\la}(\rho) &= \frac{1}{2\pi i} \int_{0}^{\la} \left[ R_{\mathbb{H}^n}\left(\sigma; \frac{n-1}{2} - i\sqrt{r}\right) - R_{\mathbb{H}^n}\left(\sigma; \frac{n-1}{2} + i\sqrt{r}\right) \right] dr \\ 
&= \frac{1}{\pi i} \int_{0}^{\sqrt{\la}} r \left[ R_{\mathbb{H}^n}\left(\sigma; \frac{n-1}{2} - ir\right) - R_{\mathbb{H}^n}\left(\sigma; \frac{n-1}{2} + ir\right) \right] dr \\
&= \frac{2^{n-3} \Gamma\left(\frac{n}{2}\right)}{\pi^{\frac{n}{2}+1}} \int_{0}^{\sqrt{\la}} \sigma_{\mathbb{H}^n}(r) \mathcal{C}_{-\frac{n-1}{2}-ir}^{\frac{n-1}{2}}(\cosh \rho) \, dr.
\end{aligned} 
\end{equation}

Now, due to the invariance of the $\Delta_{\mathbb{H}^n}$ under the isometry group, the kernel $K_{\mathbb{H}^n}(w,w')$ in 
\begin{equation} \label{6.63}
h\left( \sqrt{\Delta'_{\mathbb{H}^n}} \right) f(w) = \int_{\mathbb{H}^n} K_{\mathbb{H}^n}(w, w') f(w') \, d\mu_{\mathbb{H}^n}(w')
\end{equation}
depends only on $\rho$, $K_{\mathbb{H}^n}(w,w')=K_{\mathbb{H}^n}(\rho)$.  On the other hand, by the functional calculus, we have
\begin{equation}
\begin{aligned} \label{6.64}
h\left(\sqrt{\Delta_{\mathbb{H}^n}'}\right) f(w) 
&= \int_0^\infty h(\sqrt{\la}) \, dE_\la f(w) \\
&= \int_0^\infty h(\sqrt{\la}) \left[  \frac{1}{2\pi i}\int_0^\infty R_{\mathbb{H}^n}\left(\sigma; \frac{n-1}{2} - ir\right)  \right. \\
&\quad \left. - R_{\mathbb{H}^n}\left(\sigma; \frac{n-1}{2} + ir\right) f(\rho') \, d\rho'\right] d\la\\
&= \int_0^\infty  \left[\frac{1}{2\pi i} \int_0^\infty  h(\sqrt{\la})  R_{\mathbb{H}^n}\left(\sigma; \frac{n-1}{2} - ir\right)  \right. \\
&\quad \left. - R_{\mathbb{H}^n}\left(\sigma; \frac{n-1}{2} + ir\right) d\la \right] f(\rho') \, d\rho' 
\end{aligned}.
\end{equation}
By comparing (\ref{6.64})  with (\ref{6.63})  and using  (\ref{6.56}), we find 
\begin{equation}
\begin{aligned}
K_{\mathbb{H}^n}(\rho) &=  \frac{1}{2\pi i} \int_0^\infty h(\sqrt{\la}) \left( R_{\mathbb{H}^n}\left(\sigma; \frac{n-1}{2} - ir\right) - R_{\mathbb{H}^n}\left(\sigma; \frac{n-1}{2} + ir\right) \right) d\la \\
&= \frac{1}{\pi i} \int_0^\infty r h(r) \left( R_{\mathbb{H}^n}\left(\sigma; \frac{n-1}{2} - ir\right) - R_{\mathbb{H}^n}\left(\sigma; \frac{n-1}{2} + ir\right) \right) dr \\
&= \frac{2^{n-3} \Gamma\left(\frac{n}{2}\right)}{\pi^{\frac{n+2}{2}}} \int_0^\infty h(r) \sigma_{\mathbb{H}^n}(r) \mathcal{C}_{-\frac{n-1}{2}-ir}^{\frac{n-1}{2}}(\cosh \rho) \, dr.
\end{aligned}
\end{equation}
This proves (\ref{6.54}).

For (\ref{6.55}), Awonusika shows that 
\begin{equation*}
    \mathcal{D}_n(r) := 2^{2n-4} \pi^{-1} \Gamma\left(\frac{n}{2}\right)^2 \sigma_{\mathbb{H}^n}(r).
\end{equation*}
satisfies
\begin{equation} \label{6.59}
    \mathcal{D}_{n+2}(r) = \left[ \left(\frac{n-1}{2}\right)^2 + r^2 \right] \mathcal{D}_n(r), \; \text{for} \; n \ge 1.
\end{equation} 
Then, by differentiating both sides of (\ref{6.54}) with respect to $\rho$, 
using the formula 
\begin{equation*} 
    \left. \frac{d P_{-\frac{1}{2}+i r}^{\frac{2-n}{2}}(x)}{d x} \right|_{x=\cosh \rho} = -\frac{\left(\frac{n-1}{2}\right)^2 + r^2}{\sinh \rho} P_{-\frac{1}{2}+i r}^{-\frac{n}{2}}(\cosh \rho) + \left(\frac{n-2}{2}\right) \frac{\cosh \rho}{\sinh^2 \rho} P_{-\frac{1}{2}+i r}^{\frac{2-n}{2}}(\cosh \rho)
\end{equation*}
and  (\ref{6.59}), we get (\ref{6.55}).

Now we compute the heat kernel. We take   $h_t(r) = e^{-r^2 t} e^{-\frac{(n-1)^2}{4} t}, \quad t > 0$, in (\ref{6.54}) and, using (\ref{mathcalC}) and (\ref{planc}), we obtain

\begin{equation}\label{6.68}
K_{\mathbb{H}^n}(t,w,w')
=
\frac{e^{-\frac{(n-1)^2}{4}t}}{(2\pi)^{\frac{n}{2}}}
\sinh^{\frac{2-n}{2}}\rho
\int_0^\infty e^{-r^2 t}\mathcal{D}_n(r)
P_{-\frac12+ir}^{\frac{2-n}{2}}(\cosh\rho)\,dr.
\end{equation}
The Millson formula (\ref{6.55}) for the heat kernel now takes the form
\begin{equation} \label{6.69}
    K_{\mathbb{H}^n}(t, w, w') = -\frac{e^{-(n-2)t}}{2\pi \sinh \rho} \frac{\partial}{\partial \rho} K_{\mathbb{H}^{n-2}}(t, w, w'), \quad n \ge 3.
\end{equation}

For the computation of the heat kernel we distinguish two cases:

\textbf{Case (a):} $n$ is odd, $n\geq 1$. We substitute the formula (see \cite{Strichartz1}) 
\begin{equation}\label{strich id.}
\sinh^{\frac{2-n}{2}}\rho \,\mathcal{D}_n(r)
P_{-\frac12+ir}^{\frac{2-n}{2}}(\cosh\rho)
=
\frac{(2\pi)^{\frac{n}{2}}}{\pi}
\left(-\frac{1}{2\pi\sinh\rho}\frac{\partial}{\partial\rho}\right)^{\frac{n-1}{2}}
\cos(\rho r)
\end{equation}
into (\ref{6.68})  and get 
\begin{equation*}
K_{\mathbb{H}^{n}}(t, w, w') = \frac{e^{-\frac{(n-1)^{2}}{4} t}}{2^{\frac{n-1}{2}} \pi^{\frac{n+1}{2}}} \left( -\frac{1}{\sinh \rho} \frac{\partial}{\partial \rho} \right)^{\frac{n-1}{2}} \int_{0}^{\infty} e^{-r^{2} t} \cos \rho r \, d r.
\end{equation*}
Upon applying the identity
\begin{equation*}
\int_{0}^{\infty} e^{-\beta^{2} \eta^{2}} \cos \alpha \eta \, d \eta = \frac{\sqrt{\pi}}{2 \beta} e^{-\frac{\alpha^{2}}{4 \beta^{2}}}, \quad \alpha > 0,
\end{equation*}
we get
\begin{equation*}
K_{\mathbb{H}^{n}}(t, w, w') = \frac{e^{-\frac{(n-1)^{2}}{4} t}}{\sqrt{2}(2 \pi)^{\frac{n}{2}} \sqrt{t}} \left( -\frac{1}{\sinh \rho} \frac{\partial}{\partial \rho} \right)^{\frac{n-1}{2}} e^{-\frac{\rho^{2}}{4 t}}, 
\end{equation*}
which agrees with (\ref{6.40}).

\textbf{Case (b):} $n$ is even, $n\geq 2$. We use the identity (see \cite{Strichartz2} )
    \begin{equation}\label{strich id 2}
        \sinh^{\frac{2-n}{2}} \rho \mathcal{D}_n(r) P^{\frac{2-n}{2}}_{-\frac{1}{2}+ir}(\cosh \rho) = r \tanh \pi r \frac{(2\pi)^{\frac{n}{2}}}{2\pi} \left(-\frac{1}{2\pi \sinh \rho} \frac{\partial}{\partial \rho}\right)^{\frac{n-2}{2}} P_{-\frac{1}{2}+ir}(\cosh \rho).
    \end{equation}
By replacing the Legendre function with its integral representation (see (B.86) in \cite{Awonusika}), we have
    \begin{equation} \label{6.71}
        \sinh^{\frac{2-n}{2}} \rho \mathcal{D}_n(r) P^{\frac{2-n}{2}}_{-\frac{1}{2}+ir}(\cosh \rho) = \frac{(2\pi)^{\frac{n}{2}}\sqrt{2}}{2\pi^2} \left(-\frac{1}{2\pi \sinh \rho} \frac{\partial}{\partial \rho}\right)^{\frac{n-2}{2}} \int_{\rho}^{\infty} \frac{r \sin r u \, du}{\sqrt{\cosh u - \cosh \rho}}.
    \end{equation}
Now, using (\ref{6.71}) and the identity
    $$
    \int_{0}^{\infty} \beta e^{-b^{2}\beta^{2}} \sin a\beta \, d\beta = \frac{a\sqrt{\pi}}{4b^{3}} e^{-\frac{a^{2}}{4b^{2}}}
    $$
    
in (\ref{6.68}), we obtain the following
    \begin{equation}
        K_{H^n}(t, w, w') = \frac{e^{-\frac{n(n-2)}{4}t} e^{-\frac{t}{4}}\sqrt{2}}{8\pi^{\frac{3}{2}} t^{\frac{3}{2}}} \left(-\frac{1}{2\pi \sinh \rho} \frac{\partial}{\partial \rho}\right)^{\frac{n-2}{2}} \int_{\rho}^{\infty} \frac{u e^{-\frac{u^2}{4t}} \, du}{\sqrt{\cosh u - \cosh \rho}},
    \end{equation}
which agrees with (\ref{6.41}). This completes the proof of Theorem \ref{Thm 6.7 (0-forms)}.

\section{The heat kernel on 1-forms on $\mathbb{H}^n$}
In this section, we first show how the decomposition (\ref{dec}) can be used to determine the heat kernel on 1-forms on the hyperbolic space $\mathbb{H}^n$ for the subspace W. Next, we state and prove the analogs of Theorem \ref{Thm G. H^n 0-f.} and Theorem \ref{Thm 6.7 (0-forms)}, which give us the heat kernel. 

\subsection{Reducing the vector case to the scalar case}
Let $\{\bar{\theta}^1, \dots, \bar{\theta}^{n-1}, d \rho\}$ be the orthonormal coframe of $W$ as defined in Section \ref{section3}. By (\ref{eq1heatdec}) we find that the heat kernel on 1-forms on $\mathbb{H}^n$ 
for $W$ has the following expression
\begin{equation*}
\begin{aligned}
    \vec{p} \vert_{W}(t,x,y)= &\sum_{i,k=1}^{n-1} e_{i,k} (t,x,y)\bar{\theta}^i_x\otimes \bar{\theta}^k_y+ \sum_{i=1}^{n-1}e_{i,n}(t,x,y) d\rho_x\otimes\bar{\theta}^i_y\\
    &+\sum_{i=1}^{n-1}e_{n,i}(t,x,y)\bar{\theta}^i_x \otimes d\rho_y +e_{n,n}(t,x,y) d\rho_x \otimes d\rho_y
\end{aligned}.
\end{equation*}
\begin{prop}
With respect to the orthonormal coframe $\{\bar{\theta}^1, \dots, \bar{\theta}^{n-1}, d\rho\}$, all off-diagonal components of the heat kernel vanish, $e_{i,j}=0$ for all $i\neq j$ . The non-zero diagonal components are determined by the heat kernels of the restricted operators expressed in terms of the radial differential operators $\tilde{\D}_1$ and  $\D_2$:
\begin{enumerate}
    \item For $i \in \{1, \dots, n-1\}$, $e_{i,i}\bar{\theta}^i$ is the heat kernel of $\Delta_{\mathbb{H}^n}^{(1)}\vert_{A_i}$, which  corresponds to the action of $\tilde{\D}_1$ on the radial variable $\rho$.
    \item For $i = n$, $e_{n,n} d\rho$ is the heat kernel of $\Delta_{\mathbb{H}^n}^{(1)}\vert_{B}$, which  corresponds to the action of $\D_2$ on the radial variable $\rho$.
\end{enumerate}
\end{prop}
\begin{proof}
We show how (\ref{dec}) gives
\begin{equation}\label{dec''}
e^{-t\vec{\Delta}^{(1)}_{\mathbb{H}^n} \vert_{W}}=\bigoplus_{i=1}^{n-1}e^{-t\vec{\Delta}^{(1)}_{\mathbb{H}^n}\vert_{ A_i}} \oplus e^{-t\vec{\Delta}^{(1)}_{\mathbb{H}^n}\vert_B}
\end{equation}
and this gives the conclusion. We expand a $1$-form $\omega$ in $W$, as in Section \ref{section3}, 
\[
\omega (x)= h_1(\rho_x) \bar{\theta}_x^1 + \dots + h_{n-1}(\rho_x) \bar{\theta_x}^{n-1} + h(\rho_x) d\rho_x.
\]
Then by (\ref{dec}) we have that the heat equation
$$
    \begin{cases}
     \left( \partial_t +\vec{\Delta}^{(1)}_{\mathbb{H}^n}\vert_W\right) \omega(t,x)=0 \\
     \;\omega (x, 0)= \omega_0
    \end{cases},
$$
where 
\[
\omega(t,x) = h_1(t,\rho_x) \bar{\theta}_x^1 + \dots + h_{n-1}(t,\rho_x) \bar{\theta_x}^{n-1} + h(t,\rho_x)  d\rho_x,
\]
$$\omega_0(x)=h_1^0(\rho_x) \bar{\theta}_x^1 + \dots + h_{n-1}^0(\rho_x)\bar{\theta}_x^{n-1} + h^0(\rho_x)  d\rho_x$$
is equivalent to 
$$
    \begin{cases}
     \left( \partial_t + \vec{\Delta}^{(1)}_{\mathbb{H}^n}\vert_{ A_i}\right) \left(h_i(t,\rho_x)\bar{\theta}_x^i\right)=0 \\
     \;h_i(0,\rho_x)\bar{\theta}_x^i= h_i^0(\rho_x)\bar{\theta}_x^i,  
    \end{cases}, i=1,\dots , n-1
$$ and 
$$
    \begin{cases}
     \left( \partial_t + \vec{\Delta}^{(1)}_{\mathbb{H}^n}\vert_{ B}\right) \left( h_i(t,\rho_x)d\rho_x \right)=0 \\
     \;h_i(0,\rho_x)d\rho_x= h_i^0(\rho_x)d\rho_x  
    \end{cases}.
    $$ 
Now using this equivalence, we can show by the same procedure as in $\mathbb{R}^n$ the relation (\ref{dec''}).
\end{proof}

\begin{Remark}\label{rem. 4.2}
By Remarks \ref{eq1formszero} and \ref{eq2formszero}
 we have $\D_2=\tilde{\D}_1 +(n-2)I$. Thus, the heat kernel of $\vec{\Delta}^{(1)}_{\mathbb{H}^n}\vert_{ B}$, $\vec{K}^{(1)}_{\mathbb{H}^n} \vert_{ B} (t,\rho)$ is related to the heat kernel of $\vec{\Delta}^{(1)}_{\mathbb{H}^n}\vert_{ A_i}$, $\vec{K}^{(1)}_{\mathbb{H}^n} \vert_{ A_i} (t,\rho)$ by
 $$
\vec{K}^{(1)}_{\mathbb{H}^n} \vert_{ B} (t,\rho)=e^{-(n-2)t} \;\vec{K}^{(1)}_{\mathbb{H}^n} \vert_{ A_i} (t,\rho)
 $$
Thus, it is enough to compute the heat kernel $\vec{K}^{(1)}_{\mathbb{H}^n} \vert_{ A_i} (t,\rho)$ only.
 \end{Remark}

\subsection{The heat kernel of $\vec{\Delta}^{(1)}_{\mathbb{H}^n}\vert_{ A_i}$}

For the computation of the heat kernel of $\vec{\Delta}^{(1)}_{\mathbb{H}^n}\vert_{ A_i}$, we proceed as in Theorem 6.7 in  \cite{Awonusika} by replacing the radial Laplacian of $\mathbb{H}^n$,
$\D_{\mathbb{H}^n}^{\textbf{rad}}$ with $\tilde{\D}_1\theta^i$  as indicated by  (\ref{dec'}). The analog of Theorem \ref{Thm 6.7 (0-forms)} for  $\vec{\Delta}^{(1)}_{\mathbb{H}^n}  \vert_{ A_i}$ is as follows.
\begin{theorem} \label{Thm6.7 (1-forms)}
   Let  $\left(\vec{\Delta}^{(1)}_{\mathbb{H}^n}  \vert_{ A_i}\right)':=  \vec{\Delta}^{(1)}_{\mathbb{H}^n}  \vert_{ A_i}+ \frac{(n-3)^2}{4}$. Define the operator $h\left( \sqrt{\left(\vec{\Delta}^{(1)}_{\mathbb{H}^n}  \vert_{ A_i}\right)'} \right)$ via the functional calculus for a suitable function $h$ satisfying
\begin{equation*}
h\left( \sqrt{\left(\vec{\Delta}^{(1)}_{\mathbb{H}^n}  \vert_{ A_i}\right)'} \right) \eta(w) = \int_{\mathbb{H}^n} <\vec{K}^{(1)}_{\mathbb{H}^n} \vert_{ A_i}(w, w'), \eta(w')>_{w'} \, d\mu_{\mathbb{H}^n}(w').
\end{equation*}
Then, its kernel $\vec{K}^{(1)}_{\mathbb{H}^n} \vert_{ A_i}(w, w')= K_{1,\mathbb{H}^n} \vert_{ A_i} (\rho) \sinh \rho \ \theta^i$, is given by
\begin{equation}\label{6.54'}
   K_{1,\mathbb{H}^n} \vert_{ A_i} (\rho)= \frac{2^{n-2}\Gamma\left(\frac{n+2}{2}\right)}{\pi^{\frac{n}{2}+2}} \sinh \rho\int_0^\infty  \frac{h(r)}{\left(\left(\frac{n-1}{2}\right)^2+ r^2\right)}\sigma_{\mathbb{H}^{n+2}}(r) \mathcal{C}_{-\frac{n+1}{2}-ir}^{\frac{n+1}{2}}(\cosh \rho)dr.
\end{equation} 
\end{theorem}

As seen above, the proof of Theorem \ref{Thm 6.7 (0-forms)} is heavily based on the formula of Green's function of the Laplacian on functions on $\mathbb{H}^{n}$.  Thus, for the proof of Theorem \ref{Thm6.7 (1-forms)} we also have to obtain the Green Function of  $\vec{\Delta}^{(1)}_{\mathbb{H}^n}  \vert_{ A_i}$.

For the computation of the Green's Function of  $\vec{\Delta}^{(1)}_{\mathbb{H}^n}  \vert_{ A_i}$, we proceed as in the proof of Theorem \ref{Thm G. H^n 0-f.} with a minor change in the computation of its asymptotic behavior. This minor change will be based on the fact that the radial Laplacian, $\D_{\mathbb{H}^n}^{\textbf{rad}}$, on $\mathbb{H}^n$ is unitarily equivalent to $\tilde{\Delta}_1$  under some shift. This fact, that $\D_{\mathbb{H}^n}^{\textbf{rad}}$  is unitarily equivalent to $\tilde{\Delta}_1$, is known from Lemma 2.1 and Theorem 2.2 in \cite{donnelly1981differential} but for our purpose we have to prove it with a different method based on supersymmetric framework. 
\begin{theorem} \label{Green1-for}
The Green function $\vec{R}^{(1)}_{\mathbb{H}^n}  \vert_{ A_i} (\rho;s)$ for $\vec{\Delta}^{(1)}_{\mathbb{H}^n}  \vert_{ A_i}$  is equal to $R_1(\rho;s) \vert_{ A_i}  \sinh \rho \ \theta^i$, where
\begin{equation}\label{R1(s,r)}
   R_1( \rho;s) = c_n(s) \sigma^{-(2+s)} \sqrt{4 \sigma(\sigma-1)} \;F\left(2+s, 2 - \frac{n}{2}+s, 4-n+ 2s, \frac{1}{\sigma}\right)
\end{equation}
with $\sigma := \cosh^2(\rho/2)$ and 
\begin{equation}
    c_n(s) = \frac{1}{(s+1)(n-2-s) \cdot \pi^{n/2} \cdot 2^{n+1} \cdot }  \frac{\Gamma(2+s) \Gamma\left(2 - \frac{n}{2}+s\right)}{\Gamma(4-n+ 2s)}
\end{equation}
and $\theta^i$ is as defined in Section \ref{section3}.
\end{theorem}
\begin{proof}
By \cite{donnelly1981differential} the spectrum of  $\vec{\Delta}^{(1)}_{\mathbb{H}^n}  \vert_{ A_i}$  is $[\frac{(n-3)^2}{4}, \infty)$. Thus, for  $\vec{\Delta}^{(1)}_{\mathbb{H}^n}  \vert_{ A_i}$, by the equation (\ref{eignv eq. f.}) with $k=1$, we have to take $z=s(n-3-s)$, so the resolvent of $\vec{\Delta}^{(1)}_{\mathbb{H}^n}  \vert_{ A_i}$ is 
$$
\vec{R}^{(1)}_{\mathbb{H}^n} \vert_{ A_i}(s) := (\vec{\Delta}^{(1)}_{\mathbb{H}^n}  \vert_{ A_i} - s(n - 3- s))^{-1} , 
$$
 for  $\text{Re } s > \frac{n-3}{2} $ with  $ s \notin [\frac{n-3}{2},n-3]$ .
Similarly to the proof of Theorem \ref{Thm G. H^n 0-f.}, we have that the Green function of  $\vec{\Delta}^{(1)}_{\mathbb{H}^n}  \vert_{ A_i}$ coincides with the integral kernel $\vec{R}^{(1)}_{\mathbb{H}^n} \vert_{ A_i} (w,w',s) $ of  $ \vec{R}^{(1)}_{\mathbb{H}^n} \vert_{ A_i}(s)$. Thus, by the definition of $\vec{R}^{(1)}_{\mathbb{H}^n} \vert_{ A_i}(s)$ we have
\begin{equation*} 
(\vec{\Delta}^{(1)}_{\mathbb{H}^n}  \vert_{ A_i} - s(n - 3 - s))\vec{R}^{(1)}_{\mathbb{H}^n}  \vert_{ A_i} (w,w',s)  = \vec{\delta}^{(1)} \vert_{ A_i} (w - w').
\end{equation*}
By setting $R_{1}(w,w';s)=\tilde{f_s}(d(w,w')) $ for some function $\tilde{f_s}$ and considering geodesic polar coordinates $ (\rho, \p) $ centered on $w'$ and setting $\rho=d(w,w')$ we have $\vec{R}^{(1)}_{\mathbb{H}^n} \vert_{ A_i} (w,w';s)=\vec{R}^{(1)}_{\mathbb{H}^n}  \vert_{ A_i} (\rho;s)=\tilde{f_s}(\rho)\sinh \rho \ \theta^i$. Using this and  (\ref{dec'})  the above equation reduces to 
\begin{equation} \label{Bdr Ball 1-f.}
(\tilde{\Delta}_1- s(n - 3 - s))\tilde{f_s}(\rho) = \frac{\delta(\rho)}{\sinh^{n-1} \rho}.
\end{equation}
By Remark \ref{eq2formszero}  the  homogeneous equation corresponding to (\ref{Bdr Ball 1-f.}) is
 \begin{equation} \label{(4.8)'}
   - \tilde{f_s}(\rho)'' - (n-1) \coth \rho \, \tilde{f_s}(\rho)' + \left[ -(n-2)- s(n - 3 - s)  +\frac{ n-1}{\sinh^{2} \rho} \right] \tilde{f_s}(\rho)=0.
\end{equation}
If we make the transformation $ h_s(u) =\tilde{f_s}(\rho)$ given by $u :=\cosh{\rho}$ in (\ref{(4.8)'}) (for simplicity we write $h(u)$ for $h_s(u)$), then (\ref{(4.8)'}) becomes
\begin{equation}\label{eq1}
    (1-u^2)h''(u) - nu h'(u) + \left[ -n + 2 - s(n-3-s) - \frac{n-1}{1-u^2} \right] h(u) = 0.
\end{equation}
Since the solution to the above equation should be regular at infinity when $\text{Re} s> \frac{n-3}{2}$, the solution should be 

\begin{equation*} \label{sol.1}
    h(u)=c_n(s) (u^2 - 1)^{\frac{2 - n}{4}}P^{\frac{n}{2}}_{-2+\frac{n}{2} - s}(u),
\end{equation*}
for some coefficient $c_n(s)$. As in the (\ref{sol.}) we consider instead the solution 
\begin{equation} \label{sol.2}
    h(u)=c_n(s) (u^2 - 1)^{\frac{2 - n}{4}}Q^{-\frac{n}{2}}_{1-\frac{n}{2} + s}(u).
\end{equation}

Now, in order to determine the coefficient $c_n(s)$ we follow the same technique as in the proof of Theorem \ref{Thm G. H^n 0-f.}. First, we use the asymptotic behavior of $Q^\m_\nu(u)$ as $u \to 1^+$ given by  (\ref{Batem.(11)}) and find \begin{equation*}
     Q_{1-\frac{n}{2} + s}^{\frac{-n}{2}}(u) \sim 2^{\frac{n}{4} - 1} \frac{\Gamma\left(\frac{n}{2}\right) \Gamma\left(\nu + \mu + 1\right)}{\Gamma\left(2 + s\right)} (1 - u)^{-\frac{n}{4}}, \quad u \to 1^+
 \end{equation*}

which give us  
\begin{equation*}
    (u^2 - 1)^{\frac{2 - n}{4}}Q_{1-\frac{n}{2} + s}^{\frac{-n}{2}}(u) \sim -2^{-\frac{1}{2}} \frac{\Gamma\left(\frac{n}{2}\right) \Gamma\left(\nu + \mu + 1\right)}{\Gamma\left(2 + s\right)} (u-1)^{-\frac{n-1}{2}}, \quad u \to 1^+.
\end{equation*}
By $u =\cosh{\rho} \sim 1+\frac{\rho^2}{2}$ we have $u-1 \sim  \frac{\rho^2}{2}$, so the above equation takes the following form
\begin{equation*}
        (\sinh^2\rho)^{\frac{2 - n}{4}}Q_{1-\frac{n}{2} + s}^{-\frac{n}{2}}(\cosh \rho) \sim -2^{{\frac{n}{2}}-1} \frac{\Gamma\left(\frac{n}{2}\right) \Gamma\left(\nu + \mu + 1\right)}{\Gamma\left(2 + s\right)} \frac{1}{\rho^{n-1}}, \quad \rho \to 0^+.
\end{equation*}
Thus, the solution (\ref{sol.2}) has the asymptotic behavior
\begin{equation}\label{as. b.Q}
    \tilde{f_s}(\rho) \sim -c_n(s) 2^{{\frac{n}{2}}-1} \frac{\Gamma\left(\frac{n}{2}\right) \Gamma\left(\nu + \mu + 1\right)}{\Gamma\left(2 + s\right)} \frac{1}{\rho^{n-1}}, \; \rho \to 0^+.
\end{equation}

Here, instead of integrating (\ref{Bdr Ball 1-f.}) over the ball $B(0,\ep)$ and finding the boundary behavior, we proceed by the supersymmetric method. This method will give us the relation between the resolvent operators of $\tilde{\D}_1$ and $ \D_{\mathbb{H}^n}^{\textbf{rad}}$ (see (\ref{resolv.})), from which we obtain the relation between the corresponding Green functions (see (\ref{Green.rel.}))  and thus we can make use Theorem \ref{Thm G. H^n 0-f.}.

If in Theorem \ref{App. Thm}, we take 
$$Df(\rho):=f(\rho)'+ (n-1)\coth \rho $$
and \[
\mathcal{H}_1=\mathcal H_2=\mathcal H = L^2\big((0,\infty), \sinh^{n-1}\rho\, d\rho\big),
\]
then, since \(D\) is densely defined and closed, Hypothesis~A.1 applies, and by a straightforward calculation, we find 
$$
D^*f(\rho)=-f(\rho)', \; 
H_2=DD^*f(\rho)= \D_{\mathbb{H}^n}^{\textbf{rad}}f(\rho), \;H_1=D^*Df(\rho)=\tilde{\Delta}_1f(\rho)+(n-2)f(\rho).
$$

Furthermore, by a direct calculation we find
\[
\ker(D)=\{0\}, \qquad \ker(D^*)=\{0\}.
\]
Substituting these into part (a) of Theorem \ref{App. Thm} , we obtain
\[
\ker(H_1)=\ker(D^*D)=\{0\}, \qquad \ker(H_2)=\ker(DD^*)=\{0\},
\]
and hence
\[
\overline{\operatorname{ran}(D)}=\overline{\operatorname{ran}(D^*)}=\mathcal H.
\]
Also, substituting \(\ker(D)=\ker(D^*)=\{0\}\),  in part (b) of Theorem \ref{App. Thm} , the unitary equivalence applies without the need for kernel projections, and one concludes that
\[
\tilde{\D}_1+(n-2)I \quad \text{is unitarily equivalent to} \quad \D_{\mathbb{H}^n}^{\textbf{rad}}.
\]
Combining the above, part (c) of Theorem \ref{App. Thm} takes the following form 
\begin{equation} \label{resolv.}
    I +z (H_1 -z I)^{-1}=D^*(H_2-z I)^{-1} D,\; z \in \rho(H_1) \cap \rho (H_2),
\end{equation}
where 
$$
H_2= \D_{\mathbb{H}^n}^{\textbf{rad}}f(\rho), \;H_1=\tilde{\Delta}_1f(\rho)+(n-2)f(\rho).
$$
In order to use (\ref{resolv.}) to obtain the Green function of $(\tilde{\D}_1-s(n-3-s) I)^{-1}$ we have to take the appropriate $z$ in it. This $z$ has to satisfy 

$$
\tilde{\D}_1 +(n-2)I -z I= \tilde{\D}_1 -s(n-3-s) I,
$$
that is  
$$z =s(n-3-s) +(n-2)=-s^2+sn-3s+n-2.$$
If we shift the spectral parameter $s(n-1-s)$ of $\D_{\mathbb{H}^n}^{\textbf{rad}}$ by $s\to s+1$ then 
$$(s+1)(n-1-(s+1))=-s^2+sn-3s+n-2.$$
Thus, we take $z=(s+1)(n-1-(s+1))=(s+1)(n-2-s)$ with $s \in \mathbb{C}, \text{Re} s> \frac{n-3}{2}$ in (\ref{resolv.}) and have
\begin{equation*} 
I + (s+1)(n-2-s)(\tilde{\D}_1-s(n-3-s) I)^{-1}=D^*(\D_{\mathbb{H}^n}^{\textbf{rad}}-(s+1)(n-2-s) I)^{-1} D.
\end{equation*} 
Setting $l_0(s)=(s+1)(n-2-s)$ and $l_1(s)=s(n-3-s)$ we have  
\begin{equation} \label{SUSY res.}
I +l_0(s) (\tilde{\D}_1-l_1(s) I)^{-1}=D^*(\D_{\mathbb{H}^n}^{\textbf{rad}}-l_0(s) I)^{-1} D.
\end{equation}
In what follows, we denote by  $ R_1(\rho, \rho';s)$ and $ R_0(\rho, \rho';s+1)$ the integral kernels of $\tilde{\D}_1, \D_{\mathbb{H}^n}^{\textbf{rad}}$ and set $dw(\rho) := \sinh^{n-1} \rho \, d\rho$. For $f(\rho) \in \mathcal{H}$, we have
\begin{align*}
(\D_{\mathbb{H}^n}^{\textbf{rad}} - \ell_0(s))^{-1} (Df)(\rho) &= \int_0^\infty R_0(\rho, \rho';s+1) Df(\rho') \, dw(\rho') \\
&= \int_0^\infty D_{\rho'}^* R_0(\rho, \rho';s+1) f(\rho') \, dw(\rho')
\end{align*}
where in the last equality, we used the fact that $D^*$ is the adjoint of $D$ in $\mathcal{H}$. Applying $D^*$ we have
\[
D^* (\D_{\mathbb{H}^n}^{\textbf{rad}} - \ell_0(s))^{-1} (Df)(\rho) = \int_0^\infty D_\rho^* D_{\rho'}^* R_0(\rho, \rho';s+1) f(\rho') \, dw(\rho').
\]
Now, we substitute (\ref{SUSY res.}) into the left-hand side to get
\[
\left(I + \ell_0(s) (\tilde{\D}_1 - \ell_1(s))^{-1}\right) f(\rho) = \int_0^\infty D_\rho^* D_{\rho'}^* R_0(\rho, \rho';s+1) f(\rho') \, dw(\rho')
\]
or

\[
(\tilde{\D}_1 - \ell_1(s))^{-1} f(\rho) = -\frac{f(\rho)}{\ell_0(s)} + \frac{1}{\ell_0(s)} \int_0^\infty D_\rho^* D_{\rho'}^* R_0(\rho, \rho';s+1) f(\rho') \, dw(\rho')
\]
By
\[
(\tilde{\D}_1 - \ell_1(s))^{-1} f(\rho) = \int_0^\infty R_1(\rho, \rho';s) f(\rho') \, dw(\rho')
\]
and
\[
f(\rho) = \int_0^\infty \frac{\delta(\rho - \rho')}{\sinh^{n-1} \rho'} f(\rho') \, dw(\rho')
\]
we have
\begin{equation*}
   \int_0^\infty R_1( \rho, \rho';s) f(\rho') \, dw(\rho') 
= \frac{1}{\ell_0(s)} \int_0^\infty \left[ D_\rho^* D_{\rho'}^* R_0(\rho, \rho';s+1) - \frac{\delta(\rho - \rho')}{\sinh^{n-1} \rho'} \right] f(\rho') \, dw(\rho'),
\end{equation*}
which gives us 
\[
R_1(s; \rho, \rho') = \frac{1}{\ell_0(s)} \left[ D_\rho^* D_{\rho'}^* R_0(\rho, \rho';s+1) - \frac{\delta(\rho - \rho')}{\sinh^{n-1} \rho'} \right],
\]
where the equality holds in the sense of  distributions.

Now, recall that $w' = (\rho', \phi')$ is the centered point. Thus for $w\neq w'$ we  have  $d(w, w') = \rho >0$, which implies $\frac{\delta(\rho - \rho')}{\sinh^{n-1} \rho'} = 0$. This yields
\[
R_1( \rho;s) = \frac{1}{\ell_0(s)} D_\rho^* R_0( \rho;s+1).
\]
Thus, finally we have obtained
\begin{equation} \label{Green.rel.}
    R_1(\rho;s) = -\frac{1}{\ell_0(s)} \frac{d}{d\rho} R_0( \rho;s+1).
\end{equation}

Substituting  the asymptotic expansion of  $f_{s+1}(\rho)$ (given by equation (\ref{bdr a. b. Ball})) and $\ell_0(s)$ into the above formula, and noting that $   R_0(\rho;s)=f_s(\rho)$, we find
\begin{equation} \label{as. b. SUSY}
    R_1( \rho;s)\sim  -\frac{1}{(s+1)(n-2-s)}  \frac{\Gamma(\frac{n}{2}  )  }{2 \pi^{\frac{n}{2}}}\frac{1}{\rho^{(n-1)}}, \rho \to 0.
\end{equation}
By comparing (\ref{as. b.Q}) and (\ref{as. b. SUSY}) we get 

\begin{equation}\label{c_n(s) 1-forms}
 c_n(s) = \frac{\Gamma(2+s)}{(s+1)(n-2-s) \cdot \pi^{n/2} \cdot 2^{n/2} \cdot \Gamma(\nu+\mu+1)}.
\end{equation}

Now, to transform (\ref{sol.2}) with respect to the hypergeometric equation, we use (\ref{Bateman10}), which in our case takes the following form

\begin{equation*}
\begin{split}
Q_{s - \frac{n}{2}+1}^{- \frac{n}{2}}(u) = & \ 2^{s - \frac{n}{2}+1} \frac{\Gamma(s - \frac{n}{2} + 2) \Gamma(\nu + \mu + 1)}{\Gamma(2s - n + 4)} (u+1)^{\frac{n}{4} - 2 - s} (u-1)^{\frac{n}{4}} \\
& \cdot F\left(2+s, s - \frac{n}{2} + 2; 2s - n + 4; \frac{2}{1+u}\right)
\end{split}.
\end{equation*}
Substituting the above  into  (\ref{sol.2})  we get
\begin{equation*}
\begin{split}
h_s(u) = &c_n(s) \frac{1}{2^{1+\frac{n}{2}}} (u+1)^{\frac{1}{2}}  (u-1)^{\frac{1}{2}} \frac{2^{2+s}}{(u+1)^{2+s}} \cdot \\
&\frac{\Gamma(s - \frac{n}{2} + 2) \Gamma(\nu + \mu + 1)}{\Gamma(2s - n + 4)} \cdot F\left(2+s, s - \frac{n}{2} + 2; 2s - n + 4; \frac{2}{1+u}\right)
\end{split}
\end{equation*}
and finally substituting $c_n(s)$ from (\ref{c_n(s) 1-forms}), we have
\begin{equation*}
\begin{split}
h_s(u) = & \frac{1}{(s+1)(n-2-s) \cdot \pi^{n/2} \cdot 2^{n+1} \cdot }    (u^2-1)^{\frac{1}{2}} \left(\frac{2}{u+1}\right)^{2+s} \cdot \\
&\frac{\Gamma (2+s)\Gamma(s - \frac{n}{2} + 2) }{\Gamma(2s - n + 4)} \cdot F\left(2+s, s - \frac{n}{2} + 2; 2s - n + 4; \frac{2}{1+u}\right)
\end{split}
\end{equation*}
By a straightforward calculation, returning to the initial variable $\sigma = \cosh ^2 \frac{\rho}{2}$, and applying the following identities
$$\frac{2}{u+1}=\frac{1}{\sigma},\;\;\; (u^2-1)^{\frac{1}{2}}= \sqrt{4\sigma(\sigma-1)}=\sinh \rho$$ we get
\begin{equation*}
\begin{split}
\tilde{f_s}(\rho) = & \frac{1}{(s+1)(n-2-s) \cdot \pi^{n/2} \cdot 2^{n+1} \cdot } \sqrt{4\sigma(\sigma-1)} \left(\frac{1}{\sigma}\right)^{2+s} \cdot \\
&\frac{\Gamma (2+s)\Gamma(s - \frac{n}{2} + 2) }{\Gamma(2s - n + 4)} \cdot F\left(2+s, s - \frac{n}{2} + 2; 2s - n + 4; \frac{1}{\sigma}\right)
\end{split}.
\end{equation*}
\end{proof}

Proof of Theorem \ref{Thm6.7 (1-forms)}:

We proceed analogously to the proof of Theorem \ref{Thm 6.7 (0-forms)}. First, we compute the difference $R_1\left(\sigma; \frac{n-3}{2} - ir\right) - R_1\left(\sigma; \frac{n-3}{2} + ir\right)$.  To evaluate this difference, we record the following identities. For $s=\frac{n-3}{2} - ir$:
$$\Gamma(2+s)=\Gamma(\frac{n+1}{2} - ir),\;\;
\Gamma(2-\frac{n}{2}+s)=\Gamma(\frac{1}{2}-ir),$$
$$\Gamma(4-n+2s)=\Gamma(1-2ir), \;\;(s+1)(n-2-s)=\left(\frac{n-1}{2}\right)^2+ r^2.$$
Likewise, for $s=\frac{n-3}{2} + ir$: 

$$\Gamma(2+s)=\Gamma(\frac{n+1}{2} + ir),\;\;
\Gamma(2-\frac{n}{2}+s)=\Gamma(\frac{1}{2}+ir),$$
$$\Gamma(4-n+2s)=\Gamma(1+2ir), \;\; (s+1)(n-2-s)=\left(\frac{n-1}{2}\right)^2+ r^2.$$
Using (\ref{R1(s,r)}), we have 
\begin{equation}\label{Green.diff.1-form}
\begin{split}
R_1\left(\sigma; \frac{n-3}{2} - ir\right) - R_1\left(\sigma; \frac{n-3}{2} + ir\right) =  \frac{\sqrt{4\sigma(\sigma-1)} }{2\left(\left(\frac{n-1}{2}\right)^2+ r^2\right)  } \cdot \\ \frac{1}{ \pi^{n/2} 2^{n} } \left[ \left(\frac{1}{\sigma}\right)^{\frac{n+1}{2} - ir} \cdot 
\frac{\Gamma (\frac{n+1}{2} - ir)\Gamma(\frac{1}{2}-ir) }{\Gamma(1-2ir)} \cdot F\left(\frac{n+1}{2} - ir,\frac{1}{2}-ir; 1-2ir; \frac{1}{\sigma}\right) \right.\\ -\left.
\left(\frac{1}{\sigma}\right)^{\frac{n+1}{2} + ir} \cdot 
\frac{\Gamma (\frac{n+1}{2} + ir)\Gamma(\frac{1}{2}+ir) }{\Gamma(1+2ir)} \cdot F\left(\frac{n+1}{2} + ir, \frac{1}{2}+ir; 1+2ir; \frac{1}{\sigma}\right)   \right]
\end{split}.
\end{equation}
If we now set 
\begin{equation}
    \begin{aligned}
    G_{r,\sigma}(n):=   \frac{1}{ \pi^{\frac{n+2}{2}} 2^{n+2} } \left[ \left(\frac{1}{\sigma}\right)^{\frac{n+1}{2} - ir} \cdot 
\frac{\Gamma (\frac{n+1}{2} - ir)\Gamma(\frac{1}{2}-ir) }{\Gamma(1-2ir)} \cdot F\left(\frac{n+1}{2} - ir,\frac{1}{2}-ir; 1-2ir; \frac{1}{\sigma}\right) \right.\\ -\left.
\left(\frac{1}{\sigma}\right)^{\frac{n+1}{2} + ir} \cdot 
\frac{\Gamma (\frac{n+1}{2} + ir)\Gamma(\frac{1}{2}+ir) }{\Gamma(1+2ir)} \cdot F\left(\frac{n+1}{2} + ir, \frac{1}{2}+ir; 1+2ir; \frac{1}{\sigma}\right)   \right] 
    \end{aligned},
\end{equation}
then $G_{r,\sigma}(n-2)$  is equal to the right-hand side of (\ref{Green.differ.}).  Thus, by  (\ref{4.15}), we have that  
\begin{equation}\label{400}
    \begin{aligned}
       G_{r,\sigma}(n-2) = &\frac{\Gamma\left(\frac{n-1}{2} - ir\right) \Gamma\left(\frac{n-1}{2} + ir\right) \Gamma\left(\frac{1}{2} + ir\right) \Gamma\left(\frac{1}{2} - ir\right)}{\pi^{\frac{n}{2}} 2^{n} (-2ir) \Gamma(-2ir) \Gamma(2ir) \Gamma\left(\frac{n}{2}\right)} \cdot\\
        & F\left(\frac{n-1}{2} + ir, \frac{n-1}{2} - ir; \frac{n}{2}; 1 - \sigma\right)
    \end{aligned} 
\end{equation}
Thus, substituting (\ref{400}) (with $n$ replaced by $n+2$) into (\ref{Green.diff.1-form}), we find 
\begin{equation}
    \begin{aligned}
    R_1\left(\sigma; \frac{n-3}{2} - ir\right) - R_1\left(\sigma; \frac{n-3}{2} + ir\right) =  \frac{\sqrt{4\sigma(\sigma-1)} }{2\left(\left(\frac{n-1}{2}\right)^2+ r^2\right)  } \cdot \\
        \frac{\Gamma\left(\frac{n+1}{2} - ir\right) \Gamma\left(\frac{n+1}{2} + ir\right) \Gamma\left(\frac{1}{2} + ir\right) \Gamma\left(\frac{1}{2} - ir\right)}{2^{n+2} \pi^{\frac{n+2}{2}} (-2ir) \Gamma(-2ir) \Gamma(2ir) \Gamma\left(\frac{n+2}{2}\right)} \cdot F\left(\frac{n+1}{2} + ir, \frac{n+1}{2} - ir; \frac{n+2}{2}; 1 - \sigma\right) 
    \end{aligned}
\end{equation}
and using (\ref{6.56}) (with $n$ replaced by $n+2$) we get 
\begin{equation}
    \begin{aligned}
    R_1\left(\sigma; \frac{n-3}{2} - ir\right) - R_1\left(\sigma; \frac{n-3}{2} + ir\right) =  \frac{\sqrt{4\sigma(\sigma-1)} }{2\left(\left(\frac{n-1}{2}\right)^2+ r^2\right)  } \cdot \\
         \frac{i}{\pi^{\frac{n+2}{2}}r} 2^{n-1} \Gamma\left(\frac{n+2}{2}\right) \sigma_{\mathbb{H}^{n+2}}(r) \mathcal{C}^{\frac{n+1}{2}}_{-\frac{n+1}{2}-ir}(\cosh \rho)
    \end{aligned},
\end{equation}
which with $\sigma= \cosh^2 \frac{\rho}{2}$, takes the following form
\begin{equation}\label{6.56'}
    \begin{aligned}
    R_1\left(\cosh^2 \frac{\rho}{2}; \frac{n-3}{2} - ir\right) - R_1\left(\cosh^2 \frac{\rho}{2}; \frac{n-3}{2} + ir\right) =  \frac{\sinh \rho }{2\left(\left(\frac{n-1}{2}\right)^2+ r^2\right)  } \cdot \\
         \frac{i}{\pi^{\frac{n+2}{2}}r} 2^{n-1} \Gamma\left(\frac{n+2}{2}\right) \sigma_{\mathbb{H}^{n+2}}(r) \mathcal{C}^{\frac{n+1}{2}}_{-\frac{n+1}{2}-ir}(\cosh \rho)
    \end{aligned}.
\end{equation}

Now, let $E^{(1)}_\la, \la \in \mathbb{R}$, be the spectral projection for the self-adjoint Laplacian $\left(\vec{\Delta}^{(1)}_{\mathbb{H}^n}  \vert_{ A_i}\right)' $

\begin{equation*}
\left(\vec{\Delta}^{(1)}_{\mathbb{H}^n}  \vert_{ A_i}\right)' \eta = \int_0^\infty \la \, dE_\la^{(1)} \eta, \quad \eta \in L^2(\Lambda^1(\mathbb{H}^n)).
\end{equation*}
Using Stone's formula (see (A.15) in \cite{Awonusika}), we have

\begin{equation*}
dE_\la^{(1)} \; \eta (w) = \frac{1}{2\pi i} \left[\vec{R}^{(1)}_{\mathbb{H}^n} \vert_{ A_i}(\la+i0) - \vec{R}^{(1)}_{\mathbb{H}^n} \vert_{ A_i}(\la-i0)\right] \eta(w) \, d\la,
\end{equation*}
where $\vec{R}^{(1)}_{\mathbb{H}^n} \vert_{ A_i}$ is the resolvent operator  defined by
$$
\left(\vec{R}^{(1)}_{\mathbb{H}^n} \vert_{ A_i}(\la) \; \eta\right)(w)=\int_{\mathbb{H}^n} <\left(\vec{R}^{(1)}_{\mathbb{H}^n} \vert_{ A_i}\right)'(w, w',\la), \eta(w')>_{w'} \, d\mu_{\mathbb{H}^n}(w')
$$
with

\begin{equation*}
\left(\vec{R}^{(1)}_{\mathbb{H}^n} \vert_{ A_i}\right)'(w, w',\la) = 
\begin{cases} 
\vec{R}^{(1)}_{\mathbb{H}^n} \vert_{ A_i}(\sigma; \frac{n-3}{2} + i\sqrt{\la}), & \text{Im } \la < 0, \\
\vec{R}^{(1)}_{\mathbb{H}^n} \vert_{ A_i}(\sigma; \frac{n-3}{2} - i\sqrt{\la}), & \text{Im } \la > 0.
\end{cases} 
\end{equation*}
By the above, we have

\begin{equation*}
\begin{aligned}
(E_{\la}^{(1)} f)(w) =  \int_{\mathbb{H}^n}\frac{1}{2\pi i} \int_0^\la<  \left( \vec{R}^{(1)}_{\mathbb{H}^n} \vert_{ A_i}\right)'& (w, w',r + i0)  - \left(\vec{R}^{(1)}_{\mathbb{H}^n} \vert_{ A_i}\right)'(w, w',r - i0) ,\eta(w') >_{w'} \cdot \\  & dr\; \eta(w') \, d\mu_{\mathbb{H}^n}(w').
\end{aligned}
\end{equation*}
Let us denote its kernel by  
\begin{equation}\label{111}
 \begin{aligned} 
\vec{\pi}^{(1)}_{\la}\vert_{ A_i}(w, w') & := \frac{1}{2\pi i} \int_{0}^{\la} \left[ \left(\vec{R}^{(1)}_{\mathbb{H}^n} \vert_{ A_i}\right)'(w, w',r + i0)  - \left(\vec{R}^{(1)}_{\mathbb{H}^n} \vert_{ A_i}\right)'(w, w',r - i0) \right] dr \\
&=\left[\frac{1}{2\pi i} \int_{0}^{\la} \left[ R_1(\rho,r + i0)  - R_1(\rho,r - i0) \right] dr \right] \sinh \rho \;\theta^i,
\end{aligned}
\end{equation}
and its scalar one by $$
\pi_{1,\la}(\rho) :=\frac{1}{2\pi i} \int_{0}^{\la} \left[ R_1(\rho,r + i0)  - R_1(\rho,r - i0) \right] dr .
$$
Then, using (\ref{6.56'}), we compute
\begin{equation}\label{pirho 1-forms}
 \begin{aligned} 
\pi_{1,\la}(w, w')=\pi_{1,\la}(\rho) &= \frac{1}{2\pi i} \int_{0}^{\la} \left[ R_1\left(\sigma; \frac{n-3}{2} - i\sqrt{r}\right) - R_1\left(\sigma; \frac{n-3}{2} + i\sqrt{r}\right) \right] dr \\
&= \frac{1}{\pi i} \int_{0}^{\sqrt{\la}} r \left[ R_1\left(\sigma; \frac{n-3}{2} - ir\right) - R_1\left(\sigma; \frac{n-3}{2} + ir\right) \right] dr \\
&= \frac{\sinh \rho }{2\left(\left(\frac{n-1}{2}\right)^2+ r^2\right)  }\frac{2^{n-1} \Gamma\left(\frac{n+2}{2}\right)}{\pi^{\frac{n}{2}+2}} \int_{0}^{\sqrt{\la}} \sigma_{\mathbb{H}^{n+2}}(r) \mathcal{C}_{-\frac{n+1}{2}-ir}^{\frac{n+1}{2}}(\cosh \rho) \, dr
\end{aligned}.
\end{equation}
Now, due to the invariance of the $\vec{\Delta}^{(1)}_{\mathbb{H}^n}  \vert_{ A_i}$ under the isometry group, (see \cite{siasos2024p}), the kernel $\vec{K}^{(1)}_{\mathbb{H}^n} \vert_{ A_i}(w, w')$ in 
\begin{equation} \label{6.63'}
h\left( \sqrt{\left(\vec{\Delta}^{(1)}_{\mathbb{H}^n}  \vert_{ A_i}\right)'} \right) \eta(w) = \int_{\mathbb{H}^n} <\vec{K}^{(1)}_{\mathbb{H}^n} \vert_{ A_i}(w, w'), \eta(w')>_{w'} \, d\mu_{\mathbb{H}^n}(w')
\end{equation}
it depends only on $\rho$, with $\vec{K}^{(1)}_{\mathbb{H}^n} \vert_{ A_i}(w, w')= K_{1,\mathbb{H}^n} \vert_{ A_i} (\rho) \sinh \rho \; \theta^i$.  On the other hand, by the functional calculus, we have
\begin{equation}
\begin{aligned} \label{6.64'}
h\left( \sqrt{\left(\vec{\Delta}^{(1)}_{\mathbb{H}^n}  \vert_{ A_i}\right)'} \right) \eta(w)
&= \int_0^\infty h(\sqrt{\la}) \, dE^{(1)}_\la \eta(w) \\
&= \int_0^\infty h(\sqrt{\la}) \left[  \frac{1}{2\pi i}\int_0^\infty <\left[R_1\left(\sigma; \frac{n-3}{2} - ir\right)  \right.\right. \\
&\quad \left. \left. - R_1\left(\sigma; \frac{n-3}{2} + ir\right)\right]\sinh \rho \;\theta^i, \eta(\rho') >\, d\rho'\right] d\la \\
&= \int_0^\infty < \frac{1}{2\pi i} \int_0^\infty  h(\sqrt{\la})\left[  R_1\left(\sigma; \frac{n-3}{2} - ir\right)  \right. \\
&\quad \left. - R_1\left(\sigma; \frac{n-3}{2} + ir\right) \right] \sinh \rho \;\theta^i d\la,  \eta (\rho')> d\rho'
\end{aligned}.
\end{equation}
By comparing (\ref{6.64'}) with (\ref{6.63'}) and using (\ref{6.56'}) we have 

\begin{equation*}
\begin{aligned}
      K_{1,\mathbb{H}^n} \vert_{ A_i} (\rho)&=  \frac{1}{2\pi i} \int_0^\infty  h(\sqrt{\la})\left[  R_1\left(\sigma; \frac{n-3}{2} - ir\right)  
 - R_1\left(\sigma; \frac{n-3}{2} + ir\right) \right]d\la\\
 &=\frac{1}{\pi i} \int_0^\infty r h(r)\left[  R_1\left(\sigma; \frac{n-3}{2} - ir\right)  
 - R_1\left(\sigma; \frac{n-3}{2} + ir\right) \right]dr\\
 &=\frac{2^{n-2}\Gamma\left(\frac{n+2}{2}\right)}{\pi^{\frac{n}{2}+2}} \sinh \rho\int_0^\infty  \frac{h(r)}{\left(\frac{n-1}{2}\right)^2+ r^2}\sigma_{\mathbb{H}^{n+2}}(r) \mathcal{C}_{-\frac{n+1}{2}-ir}^{\frac{n+1}{2}}(\cosh \rho)dr
\end{aligned},
\end{equation*}
and this completes the proof of  the theorem.

Now, using  Theorem \ref{Thm6.7 (1-forms)}, we compute the heat kernel of $\vec{\Delta}^{(1)}_{\mathbb{H}^n}  \vert_{ A_i}$.
\begin{Corollary}\label{heatkern.1-forms}
The heat kernel $\vec{K}^{(1)}_{\mathbb{H}^n} \vert_{ A_i} (t,\rho)=K_{1,\mathbb{H}^n} \vert_{ A_i}  (t, \rho)  \sinh \rho \ \theta^i$ of  $\vec{\Delta}^{(1)}_{\mathbb{H}^n}  \vert_{ A_i}$ is given by 
 \begin{itemize}
     \item [(a)] $n$ odd, $n \geq 1$:
     \begin{equation*} 
K_{1,\mathbb{H}^n} \vert_{ A_i}  (t, \rho)  =
\frac{\sinh  \rho \; e^{-\frac{(n-3)^2}{4}t} }{(2\pi)^{\frac{n+2}{2}}2^\frac{3}{2}\sqrt{t}}  
\left(-\frac{1}{\sinh\rho}\frac{\partial}{\partial\rho}\right)^{\frac{n+1}{2}}e^{-\frac{\rho^2}{4t}},
\end{equation*}
\item [(b)] $n$ even, $n \geq 2$:
$$
K_{1,\mathbb{H}^n} \vert_{ A_i}  (t, \rho)  
= \frac{\sinh \rho\; e^{-\frac{(n-3)^2}{4}t}}{8\sqrt{2}\pi^{\frac{3}{2}}t^{\frac{3}{2}}} \left(-\frac{1}{2\pi \sinh \rho} \frac{\partial}{\partial \rho}\right)^{\frac{n}{2}} \int_{\rho}^{\infty} \frac{u e^{-\frac{u^2}{4t}} \, du}{\sqrt{\cosh u - \cosh \rho}}.
$$
 \end{itemize}
Moreover, it satisfies the following Millson-type recurrence formula 
\begin{equation}\label{Millson.1-form}
     K_{1,\mathbb{H}^n} \vert_{ A_i}  (t, \rho)  = -\frac{e^{-(n-4)t}}{2\pi} \frac{\partial}{\partial \rho} \left[ \frac{K_{1, \mathbb{H}^{n-2}} \vert_{ A_i} (t, \rho)}{\sinh \rho} \right].
\end{equation}
\end{Corollary}
\begin{proof}
We take $h(r) =\left(\left(\frac{n-1}{2}\right)^2+ r^2\right)e^{-r^2t}e^{-\frac{(n-3)^2}{4}t}$ in  (\ref{6.54'}) and find with the help of (\ref{6.68}) (with $n$ replaced by $n+2$) that
\begin{equation}\label{6.68'}
 K_{1,\mathbb{H}^n} \vert_{ A_i} (t,\rho)
=
\frac{1}{2(2\pi)^{\frac{n+2}{2}}} \sinh \rho e^{-\frac{(n-3)^2}{4}t}
(\sinh\rho)^{-\frac{n}{2}}
\int_0^\infty e^{-r^2 t}\mathcal{D}_{n+2}(r)
P^{-\frac{n}{2}}_{-\frac12+ir}(\cosh\rho)\,dr.
\end{equation}
Now we consider two cases:
\textbf{Case (a)} $n$ is odd, $n\geq 1$.

First, we write (\ref{strich id.}) (with $n$ replaced by $n+2$) 
\begin{equation*}
\sinh^{-\frac{n}{2}}\rho \,\mathcal{D}_{n+2}(r)
P_{-\frac12+ir}^{-\frac{n}{2}}(\cosh\rho)
=
\frac{(2\pi)^{\frac{n+2}{2}}}{\pi (2\pi)^{\frac{n+1}{2}}}
\left(-\frac{1}{\sinh\rho}\frac{\partial}{\partial\rho}\right)^{\frac{n+1}{2}}
\cos(\rho r)
\end{equation*}
and then we substitute it in (\ref{6.68'}) and find
\begin{equation*}\
 K_{1,\mathbb{H}^n} \vert_{ A_i} (t,\rho)
=
\frac{1}{(2\pi)^{\frac{n+3}{2}}} \sinh  \rho \ e^{-\frac{(n-3)^2}{4}t}
\left(-\frac{1}{\sinh\rho}\frac{\partial}{\partial\rho}\right)^{\frac{n+1}{2}}\int_0^\infty e^{-r^2 t} \cos(\rho r) \
dr.
\end{equation*}
Using the Gaussian integral
\[
\int_0^\infty e^{-r^2 t}\cos(\rho r)\,dr
=
\frac{\sqrt{\pi}}{2\sqrt{t}}e^{-\frac{\rho^2}{4t}},
\]
we finally have
   \begin{equation*}\
\frac{\sinh  \rho \; e^{-\frac{(n-3)^2}{4}t} }{(2\pi)^{\frac{n+2}{2}}2^\frac{3}{2}\sqrt{t}}  
\left(-\frac{1}{\sinh\rho}\frac{\partial}{\partial\rho}\right)^{\frac{n+1}{2}}e^{-\frac{\rho^2}{4t}}.
\end{equation*}

\textbf{Case (b)} $n$ is even, $n\geq 2$.

Similarly to Case (a), first write (\ref{strich id 2}) (with $n$ replaced by $n+2$) 
    \begin{equation*}
        \sinh^{-\frac{n}{2}} \rho \mathcal{D}_{n+2}(r) P^{-\frac{n}{2}}_{-\frac{1}{2}+ir}(\cosh \rho) = r \tanh \pi r \frac{(2\pi)^{\frac{n+2}{2}}}{2\pi} \left(-\frac{1}{2\pi \sinh \rho} \frac{\partial}{\partial \rho}\right)^{\frac{n}{2}} P_{-\frac{1}{2}+ir}(\cosh \rho).
    \end{equation*}
By replacing the Legendre function with its integral representation, we have
    \begin{equation} \label{6.71'}
        \sinh^{-\frac{n}{2}} \rho \mathcal{D}_{n+2}(r) P^{-\frac{n}{2}}_{-\frac{1}{2}+ir}(\cosh \rho) = \frac{(2\pi)^{\frac{n+2}{2}}\sqrt{2}}{2\pi^2} \left(-\frac{1}{2\pi \sinh \rho} \frac{\partial}{\partial \rho}\right)^{\frac{n}{2}} \int_{\rho}^{\infty} \frac{r \sin r u \, du}{\sqrt{\cosh u - \cosh \rho}}.
    \end{equation}
Now using (\ref{6.71'}) and 
    $$
    \int_{0}^{\infty} \beta e^{-b^{2}\beta^{2}} \sin a\beta \, d\beta = \frac{a\sqrt{\pi}}{4b^{3}} e^{-\frac{a^{2}}{4b^{2}}}
    $$
in (\ref{6.68'}) we obtain the following
    \begin{equation}
        K_{1,\mathbb{H}^n} \vert_{ A_i} (t,\rho)= \frac{\sinh \rho\; e^{-\frac{(n-3)^2}{4}t}}{8\sqrt{2}\pi^{\frac{3}{2}}t^{\frac{3}{2}}} \left(-\frac{1}{2\pi \sinh \rho} \frac{\partial}{\partial \rho}\right)^{\frac{n}{2}} \int_{\rho}^{\infty} \frac{u e^{-\frac{u^2}{4t}} \, du}{\sqrt{\cosh u - \cosh \rho}}.
    \end{equation}
For the Millson formula, we express first $  K_{1,\mathbb{H}^n} \vert_{ A_i} (t,\rho)$ in terms of $ K_{\mathbb{H}^{n+2}}(t,\rho)$. This follows, by comparing (\ref{6.68}) (with $n$ replaced by $n+2$) with (\ref{6.68'})
\begin{equation}\label{(4)}
    K_{1, \mathbb{H}^n}(t, \rho) = \frac{1}{2}\sinh \rho\; e^{2(n-1)t}  K_{ \mathbb{H}^{n+2}}(t, \rho).
\end{equation}
Next, by comparing (\ref{(4)})  (with $n$ replaced by $n-2$) with (\ref{6.69}) (with $n$ replaced by $n+2$), we obtain the Millson formula (\ref{Millson.1-form}) for $  K_{1,\mathbb{H}^n} \vert_{ A_i} (t,\rho)$.
\end{proof}

Before presenting our main estimate, we state the Davies-Mandouvalos estimates.
\begin{theorem}[Theorem 3.1 in \cite{davies1988heat}]
If $n \geqslant 1$ is any integer, then
\begin{equation} \label{D-V}
    K_{\mathbb{H}^{n+1}}(t, \rho) \sim t^{-\frac{1}{2}(n+1)} e^{-\frac{1}{2}n^2 t - \rho^2 / 4t - \frac{1}{2}n\rho} \cdot (1 + \rho + t)^{\frac{1}{2}n - 1}(1 + \rho),
\end{equation}
uniformly for $0 \leqslant \rho < \infty$ and $0 < t < \infty$. Also, for each fixed $\rho \geqslant 0$ and $n \geqslant 1$,
\begin{equation}\label{D-V1}
    K_{\mathbb{H}^{n+1}}(t, \rho) \sim t^{-\frac{3}{2}} e^{-\frac{1}{2}n^2 t} \quad \text{as } t \to \infty.
\end{equation}
\end{theorem}
As a consequence of Corollary \ref{heatkern.1-forms} and the above theorem, we prove the following estimates for $\vec{\Delta}^{(1)}_{\mathbb{H}^n}  \vert_{ A_i}$.
\begin{Corollary} \label{D-V EST. 1-FORMS}
If $n \geqslant 1$ is any integer, then
\begin{equation*}
K_{1,\mathbb{H}^n}\vert_{ A_i} (t,\rho)  \sim t^{-\frac{n+2}{2}} \rho \, (1+\rho+t)^{\frac{n-1}{2}}   \cdot e^{ -\frac{(n-3)^2}{4}t - \frac{\rho^2}{4t} - \frac{n-1}{2}\rho },
\end{equation*}
uniformly for $0 \leqslant \rho < \infty$ and $0 < t < \infty$, also for each fixed $\rho > 0$ and $n \geqslant 1$,$$K_{1,\mathbb{H}^n} \vert_{ A_i} (t,\rho) \sim t^{-\frac{3}{2}} e^{-\frac{(n-3)^2}{4} t} \quad \text{as } t \to \infty.$$
\end{Corollary}
\begin{proof}
By (\ref{(4)}) and (\ref{D-V}) (with $n$ replaced by $n+2$) we get
\begin{equation*}
 K_{1,\mathbb{H}^n} \vert_{ A_i} (t,\rho) \sim t^{-\frac{n+2}{2}} \sinh\rho \, (1+\rho+t)^{\frac{n-1}{2}} (1+\rho) e^{ -\frac{(n-3)^2}{4}t - \frac{\rho^2}{4t} - \frac{n+1}{2}\rho }.
\end{equation*}
Now, the above equation with
\[
\sinh \rho \, e^{-\frac{n+1}{2} \rho} = \frac{1}{2} \left( 1 - e^{-2\rho} \right) e^{-\frac{n-1}{2} \rho}
\]
and
$$1 - e^{-\rho} \sim \frac{\rho}{1+\rho},$$
gives the first estimate. For the second, we fix $\rho > 0$. As $t \to \infty$,
\[
(1 + \rho + t)^{\frac{n-1}{2}} \sim t^{\frac{n-1}{2}},
\]
while
\[
e^{-\rho^2/(4t)} \to 1.
\]
Hence, the power of $t$ becomes
\[
t^{-\frac{n+2}{2}} t^{\frac{n-1}{2}} = t^{-\frac{3}{2}}.
\]
    
    \end{proof}

\section{Appendix}

\subsection{Supersymmetric framework}
The following definitions and results are adapted from the appendix in \cite{Supersymmetry}.

\noindent \textbf{Hypothesis A.1.} \textit{Let $\mathcal{H}_j, j = 1, 2$, be separable complex Hilbert spaces and}
\begin{equation*}
D : \mathcal{H}_1 \supseteq \text{dom}(D) \to \mathcal{H}_2 
\end{equation*}
\textit{be a densely defined, closed, linear operator.}

Let us introduce the associated operators: 
\begin{equation*}
 H_1 = D^*D \text{ in } \mathcal{H}_1, \quad H_2 = DD^* \text{ in } \mathcal{H}_2.
\end{equation*}
As established below, the spectral properties of any single operator in the triple $D, H_1, H_2$ essentially determine the corresponding properties of the remaining pair. Specifically: 
\begin{theorem} \label{App. Thm}
 Assume Hypothesis A.1, then the following hold
\begin{itemize}
    \item[(a)] $  \ker(D) = \ker(H_1) = (\text{ran}(D^*))^\perp, \quad \ker(D^*) = \ker(H_2) = (\text{ran}(D))^\perp. $
\item[(b)]  $H_1$ and $H_2$ are \textit{essentially isospectral, that is,}
\begin{equation*}
\sigma(H_1) \setminus \{0\} = \sigma(H_2) \setminus \{0\}, 
\end{equation*}
\textit{in fact,}
\begin{equation*}
D^*D [I_{\mathcal{H}1} - P_{\text{ker}(D)}] \text{ is unitarily equivalent to } DD^* [I_{\mathcal{H}2} - P_{\text{ker}(D^*)}].
\end{equation*}
\item[(c)] For $z \in \rho(H_1) \cap \rho(H_2)$,
\begin{align*}
I_{\mathcal{H}_2} + z(H_2 - zI_{\mathcal{H}_2})^{-1} &\supseteq D(H_1 - zI_{\mathcal{H}_1})^{-1} D^*, \\
I_{\mathcal{H}1} + z(H_1 - zI{\mathcal{H}_1})^{-1} &\supseteq D^*(H_2 - zI_{\mathcal{H}_2})^{-1} D.
\end{align*}
\end{itemize}
\end{theorem}

\bibliographystyle{alpha}
\bibliography{sample}

\end{document}